\documentclass[a4paper,10pt]{amsart}
\usepackage[utf8]{inputenc}
\usepackage{amsmath,amsthm,mathtools,amssymb, dsfont,bm}
\usepackage{hyperref,xparse}
\usepackage{a4wide,cite}
\usepackage{enumerate,moreenum}
\usepackage{multicol,breqn,enumerate}
\usepackage{pictexwd,dcpic,aliascnt}
\usepackage{tikz}
\usetikzlibrary{arrows,positioning}

\newcommand\bref[3][blue]{%
    \begingroup%
    \hypersetup{linkcolor=#1}%
    \hyperlink{#2}{#3}%
    \endgroup}

\newtheorem{theorem}{Theorem}[section]

\newtheorem{thmintro}{Theorem}

\newaliascnt{lemma}{theorem}
\newtheorem{lemma}[lemma]{Lemma}
\aliascntresetthe{lemma}

\newaliascnt{fact}{theorem}

\aliascntresetthe{fact}

\newaliascnt{corollary}{theorem}
\newtheorem{corollary}[corollary]{Corollary}
\aliascntresetthe{corollary}

\newaliascnt{proposition}{theorem}
\newtheorem{proposition}[proposition]{Proposition}
\aliascntresetthe{proposition}

\newaliascnt{df}{theorem}
\theoremstyle{definition}\newtheorem{df}[df]{Definition}
\aliascntresetthe{df}

\newaliascnt{remark}{theorem}
\theoremstyle{remark}\newtheorem{remark}[remark]{Remark}
\aliascntresetthe{remark}

\numberwithin{equation}{section}
\DeclareMathOperator{\Ob}{Ob}
\DeclareMathOperator{\Mor}{Mor}
\DeclareMathOperator{\Hom}{Hom}

\DeclareMathOperator{\Rep}{Rep}

\DeclareMathOperator{\id}{id}

\newcommand{\GG}{\mathbb{G}}
\newcommand{\HH}{\mathbb{H}}
\newcommand{\KK}{\mathbb{K}}
\newcommand{\ww}{\mathrm{W}}
\newcommand{\WW}{{\mathds{V}\!\!\text{\reflectbox{$\mathds{V}$}}}}
\newcommand{\Ww}{\mathds{W}}
\newcommand{\wW}{\text{\reflectbox{$\Ww$}}\:\!}
\newcommand{\isimplied}{\text{\reflectbox{$\implies$}}\:\!}
\newcommand{\bigslant}[2]{{\raisebox{.2em}{$#1$}\left/\raisebox{-.2em}{$#2$}\right.}}

\newcommand{\wc}{{\text{---}\scriptscriptstyle{weak^*}}}
\newcommand\vnt{\bar{\otimes}}
\newcommand\comp{\!\circ\!}
\newcommand\BLtwo{\mathsf{B}(L^2(\mathbb{G}))}
\DeclareRobustCommand\flip{\bm\sigma}

\DeclareDocumentCommand\mult{ m g g}{%
    {
    \IfNoValueTF {#3} {
    \IfNoValueT {#2} {\mathnormal{M}({#1})} 
    \IfNoValueF {#2} {\mathnormal{M}({#1}\otimes{#2})}}
    {\mathnormal{M}({#1}\otimes{#2}\otimes{#3})} }}
    
\DeclareDocumentCommand\B{ m g }{%
    {\IfNoValueT {#2} {\mathsf{B}(\mathcal{#1})} 
    \IfNoValueF {#2} {\mathsf{B}(\mathcal{#1}_{#2})} }
}
\DeclareDocumentCommand\K{ m g }{%
    {\IfNoValueT {#2} {\mathnormal{K}(\mathcal{#1})} 
    \IfNoValueF {#2} {\mathnormal{K}(\mathcal{#1}_{#2})} }
}
\DeclareDocumentCommand\Cnot{ m g }{%
    {\IfNoValueT {#2} {C_0(\mathbb{#1})} 
    \IfNoValueF {#2} {C_0(\mathbb{#1}_{#2})} }
}
\DeclareDocumentCommand\Cnothat{ m g }{%
    {\IfNoValueT {#2} {C_0(\widehat{\mathbb{#1}})} 
    \IfNoValueF {#2} {C_0(\widehat{\mathbb{#1}}_{#2})} }
}
\DeclareDocumentCommand\Cunot{ m g }{%
    {\IfNoValueT {#2} {C^u_0(\mathbb{#1})} 
    \IfNoValueF {#2} {C^u_0(\mathbb{#1}_{#2})} }
}
\DeclareDocumentCommand\Cunothat{ m g }{%
    {\IfNoValueT {#2} {C^u_0(\widehat{\mathbb{#1}})} 
    \IfNoValueF {#2} {C^u_0(\widehat{\mathbb{#1}}_{#2})} }
}
\DeclareDocumentCommand\Linf{ m g }{%
    {\IfNoValueT {#2} {L^{\infty}(\mathbb{#1})} 
    \IfNoValueF {#2} {L^{\infty}(\mathbb{#1}_{#2})} }
}
\DeclareDocumentCommand\Linfhat{ m g }{%
    {\IfNoValueT {#2} {L^{\infty}(\widehat{\mathbb{#1}})} 
    \IfNoValueF {#2} {L^{\infty}(\widehat{\mathbb{#1}}_{#2})} }
}
\DeclareDocumentCommand\Lone{ m g }{%
    {\IfNoValueT {#2} {L^{1}(\mathbb{#1})} 
    \IfNoValueF {#2} {L^{1}(\mathbb{#1}_{#2})} }
}
\DeclareDocumentCommand\Lonehat{ m g }{%
    {\IfNoValueT {#2} {L^{1}(\widehat{\mathbb{#1}})} 
    \IfNoValueF {#2} {L^{1}(\widehat{\mathbb{#1}}_{#2})} }
}

\title{Hopf Images in Locally Compact Quantum Groups}
\author{Pawe{\l} J\'oziak}
\address{Institute of Mathematics of the Polish Academy of Sciences and  Institute of mathematics, University of Wroc{\l}aw, Poland}
\email{pjoziak@impan.pl}

\author{Pawe{\l} Kasprzak}
\address{Department of Mathematical Methods in Physics, Faculty of Physics, University of Warsaw, Poland}
  \email{pawel.kasprzak@fuw.edu.pl}

\author{Piotr M.~So{\l}tan} 
\address{Department of Mathematical Methods in Physics, Faculty of Physics, University of Warsaw, Poland}
\email{piotr.soltan@fuw.edu.pl}



\begin{document}

\begin{abstract}
This manuscript is devoted to the study of the concept of a generating subset (a.k.a. Hopf image of a morphism) in the setting of locally compact quantum groups. The aim of this paper is to provide an accurate description of the Hopf image of a given morphism. We extend and unify the previously existing approaches for compact and discrete quantum groups and present some results that can shed light on some local perspective in the theory of quantum groups. In particular, we provide a characterization of fullness of Hopf image in the language of partial actions as well as in representation-theoretic terms in the spirit of representation $C^*$-categories, extending some known results not only to the broader setting of non-compact quantum groups, but also encompassing a broader setting of generating subsets.
\end{abstract}
\maketitle

\section*{Introduction}
The concept of a subgroup is central to understanding locally compact quantum groups as group-theoretic objects and accordingly it has been receiving an increasing interest in recent years, see e.g. \cite{BB09,BY14,BCV,DKSS,KKSopen,KSS,KS14b,KS15}. It arises naturally when studying quantum homogeneous spaces, as the natural examples of homogeneous spaces are of quotient type, see \cite{KS14a}, or when studying the behavior of quantum groups with respect to some natural constructions, see e.g. \cite{KSS,Ver04,Wang95, CHK16}. 

The starting point for this project was to improve the understanding of the notion of closed subgroup in the realm of locally compact quantum groups. Classically, in order for a non-empty closed subset $X\subset G$ of a locally compact group $G$, to be a group, it has to satisfy $X^2\subset X$ and, if $G$ is non-compact, also $X^{-1}=X$. Clearly, generically these conditions do not hold, and the way to cope with this is to consider the \emph{subgroup generated by the given subset $X$}, i.e.~the smallest closed subgroup containing it. One immediately realizes that this smallest subgroup is given by $\overline{\bigcup_{n\in\mathbb{Z}}X^n}$ (and in the compact case $\mathbb{Z}$ may be replaced by $\mathbb{N}$). 

This viewpoint was already utilized in \cite{BB10} in the case of Hopf ${}^*$-algebras and in \cite{SS16,BCV,ChirvaRFD} in the case of compact quantum groups, where analytic issues are minor, and it was implicit in various writings on discrete quantum groups, and formulated explicitly e.g. in \cite{Iz02,Ver05,Ver07}. Once the notion of generation was established in a satisfactory way in this restricted setting, the question of moving to a broader, and technically more involved, class of locally compact quantum groups arises naturally.

In this manuscript we study the following concept: let $\GG$ be a locally compact quantum group (in the sense of Kustermans-Vaes). Let $\mathnormal{B}$ be a $C^{\ast}$-algebra and let $\beta\in \Mor(\Cunot{G},\mathnormal{B})$ be a morphism of $C^*$-algebras (in the sense of \cite{SLW95a}). We think of it as the Gelfand dual of a map $\widehat{\beta}\colon\mathbb{X}\to\GG$ from a quantum space into a quantum group and ask what is the closed quantum subgroup (in the sense of Vaes, see \cite{DKSS}) of $\GG$ generated by $\widehat{\beta}(\mathbb{X})\subset\GG$.

Formally speaking, we consider the following category, which we denote by $\mathcal{C}_{\beta}$. Objects of $\mathcal{C}_{\beta}$ are triples $(\pi, \HH, \tilde{\beta})$ consisting of: a closed quantum subgroup $\HH$ of $\GG$ such that $\pi\in \Mor(\Cunot{G},\Cunot{H})$ is the associated morphism intertwining the coproducts and $\tilde{\beta}\comp\pi=\beta$ (as morphisms of $C^*$-algebras), i.e.~the map $\beta$ factors through the $C^*$-algebra of functions on the subgroup $\Cunot{H}$ (where $\HH$ is embedded into $\GG$ using $\pi$) and $\tilde{\beta}\comp\pi=\beta$ is the factorization. For two objects $\mathds{h}=(\pi, \HH, \tilde{\beta}),\mathds{k}=(\pi', \KK, \beta')\in \Ob(\mathcal{C}_{\beta})$, a morphism $\varphi\in \Mor_{\mathcal{C}_{\beta}}(\mathds{h},\mathds{k})$ is a $C^*$-algebra morphism $\varphi\in \Mor(\Cunot{K},\Cunot{H})$ (note that the direction of arrows in $\mathcal{C}_{\beta}$ is same as the direction of maps on the level of quantum groups, which is opposite to the ones on the level of $C^*$-algebras of functions), which intertwines the respective coproducts and such that the following diagram commutes:

{\centering
\begin{tikzpicture}
  [bend angle=36,scale=2,auto,
pre/.style={<-,shorten <=1pt,semithick},
post/.style={->,shorten >=1pt,semithick}]
\node (G) at (-1,1) {$\Cunot{G}$};
\node (B) at (1,1) {$\mathnormal{B}$}
edge [pre] node[auto,swap] {$\beta$} (G);
\node (H) at (0,0) {$\Cunot{H}$}
edge [pre] node[auto,swap] {$\pi$} (G)
edge [post] node[auto] {$\tilde{\beta}$} (B);
\node (K) at (-1,-1) {$\Cunot{K}$}
edge [pre] node[auto,swap] {$\pi'$} (G)
edge [post, bend right] node[auto,swap] {$\tilde{\beta'}$} (B)
edge [post] node[auto] {$\varphi$} (H);
 \end{tikzpicture}\vskip1em\centering\hypertarget{diagram}{Diagram 1: Morphisms in $\mathcal{C}_{\beta}$}\vskip1em}
The object we are interested in is the initial object of the category $\mathcal{C}_{\beta}$. Our study of Hopf images begins in \autoref{sec:hopfimageconstruction} by showing, among others, that
\begin{thmintro}\label{thmintro:existence}
 Given a locally compact quantum group $\GG$ and a morphism $\beta$ as above, there always exists the initial object of the category $\mathcal{C}_{\beta}$.
\end{thmintro}
This initial object is then called the Hopf image of $\beta$, following the situation studied first in the algebraic setting by T. Banica and J. Bichon in \cite{BB10}. When this initial object happens to be $(\id,\GG,\beta)$, we then say that $\beta$ is a \emph{generating morphism} (this was earlier known as \emph{inner faithful morphism} and as \emph{morphism faithful in the discrete quantum group sense} in the context of Hopf ${}^*$-algebras, but the latter term is long and does not extend nicely beyond the compact quantum group setting, and the former might unnecessarily suggest some connections to the adjoint action).

The prominent examples of induction from subgroups of $\GG$ to $\GG$ of the following properties: Connes' embeddability of $\Linf{G}$, given in \cite{BCV}, and residual fininite dimensionality of $\mathcal{O}(\GG)$, given in \cite{ChirvaRFD}, puts the setting of quantum group generated by a family of its subgroups in the center of the study of Hopf images. We analyze an analogous situation in \autoref{sec:twosubgroups}. 

Intriguingly, the closed quantum subgroups $\HH\subset\GG$ are characterized by the so-called Baaj-Vaes subalgebras of $\Linfhat{G}$ (the precise definitions are given in \autoref{sec:homosubgr}). It follows from our discussion (cf. \autoref{thm:BVclosure}) that the most natural examples of invariant subalgebras of $\Linfhat{G}$ -- the ones coming from representations of $\GG$ -- satisfy the following phenomenon: if they are $\tau$-preserved (which is the case e.g. for quantum groups that are compact, discrete or of Kac type), they are automatically Baaj-Vaes. This means that in many examples part of the assumptions of the Baaj-Vaes theorem are actually implied by sole invariance. It is not known to the authors whether this phenomenon holds in full generality.

The Hopf image construction highlights a certain local perspective in the theory of quantum groups. This rather vague idea can be forged into concrete statements. For instance, classically a ho\-mo\-mor\-phism $G\to K$ is uniquely determined by its values on the generating set. A similar statement can be proved in the quantum case:

\begin{thmintro}\label{thmintro:separation}
With the notation as above, assume that $\beta$ is a generating morphism and consider two ho\-mo\-mor\-phisms $\GG\to\KK$ described by $\varphi,\tilde{\varphi}\in\Mor(\Cunot{K},\Cunot{G})$. If $\beta\comp\varphi=\beta\comp\tilde{\varphi}$, then $\varphi=\tilde{\varphi}$. 
\end{thmintro}

For $\GG$ discrete the converse is also true in the following sense: if a quantum subset has the property that ho\-mo\-mor\-phisms are uniquely determined by the values they attain on this set, it is generating.

We address the problem of describing the Hopf image of a given morphism by means of other objects related to it. Associated to the map $\beta$, philosophically being the Gelfand dual of a map $\mathbb{X}\to\GG$, there are: restrictions of representations of $\GG$ to $\mathbb{X}$; \emph{partial action} by right shifts $\mathbb{X}\curvearrowright\GG$ and a representation of the dual quantum group $X\in\Linfhat{G}\vnt\B{H}$. We have the following result expressing generation in these terms:
\begin{thmintro}\label{thmintro:betarestriction}With the notation as above, consider the following four statements:
  \begin{enumerate}[label={\normalfont(\roman*)}]
  \item \label{it:betares1} The morphism $\beta$ is generating.
  \item \label{it:betares2} $\{(\id\otimes\omega)X \mid \omega\in \mathnormal{B}^*\}''=\Linfhat{G}$.
  \item \label{it:betares3} The partial action $\mathbb{X}\curvearrowright\GG$ is ergodic.
  \item \label{it:betares4} any two distinct representations of $\GG$, when restricted to $\mathbb{X}$, remain distinct.
   \end{enumerate}
 We have \ref{it:betares1}$\isimplied$\ref{it:betares2}$\iff$\ref{it:betares3}$\iff$\ref{it:betares4}. Moreover, \ref{it:betares1}$\implies$\ref{it:betares2} provided that $\GG$ is compact or discrete or $\tau^u_t(\ker(\beta))\subseteq\ker(\beta)$ for all $t\in\mathbb{R}$ (in particular, if $\GG$ is of Kac type).
\end{thmintro}
The ingredients of \autoref{thmintro:betarestriction} are described in detail in course of constructing the Hopf image.

We later move to the study of generation-type questions in the language of representation category of the quantum group in question. Assume $\{\HH_i\}_{i\in I}$ is a family of closed subgroups of $\GG$ and denote by $U^{\HH_i}$ the restriction of a representation $U$ of $\GG$ to $\HH_i$. Recall that $\Hom_{\GG}(U,\tilde{U})$ is the set of intertwiners between $U$ and $\tilde{U}$ as representations of $\GG$. We have

\begin{thmintro}\label{thmintro:promofintertwin}
 $\GG$ is generated by $(\HH_i)_{i\in I}$ if and only if for all pairs of representations $U,\tilde{U}$ of $\GG$ we have that
 \[\Hom_{\GG}(U,\tilde{U})=\bigcap_{i\in I}\Hom_{\HH_i}(U^{\HH_i},\tilde{U}^{\HH_i})\]
\end{thmintro}

The manuscript is organized as follows: \autoref{sec:preliminaries} is devoted to establishing notation and conventions. We collect also from various places in the literature some ingredients of the theory of locally compact quantum groups that are most relevant to our constructions. In \autoref{sec:hopfimageconstruction} we deal with \autoref{thmintro:existence} and reveal the structure of objects useful in the later study of Hopf images. In \autoref{sec:twosubgroups} we analyze the case of the generating set coming from a family of subgroups. \autoref{sec:separation} is devoted to the proof of \autoref{thmintro:separation} and its converse. In \autoref{sec:betarestriction} we introduce the concept of restricting a representation to a subset and prove a reformulation of \autoref{thmintro:betarestriction}. In \autoref{sec:intertwiners} we study the concept of restriction of a representation to a subset from the point of view of intertwiners. In particular, we derive \autoref{thmintro:promofintertwin}. \autoref{sec:examples} discusses some examples and applications of the Hopf image construction.

\section{Preliminaries}\label{sec:preliminaries}
\subsection{\texorpdfstring{$C^{\ast}$}{Cst}-algebras, von Neumann algebras etc}
The symbol $\flip$ denotes the flip, i.e.~unique extension of the map $\mathnormal{A}\otimes\mathnormal{B}\ni a\otimes b\mapsto b\otimes a\in\mathnormal{B}\otimes\mathnormal{A}$. We use the leg numbering notation, which is now commonly understood: for $T\in\mult{\mathnormal{A}}{\mathnormal{C}}$, we denote by $T_{13}\in\mult{\mathnormal{A}}{\mathnormal{B}}{\mathnormal{C}}$ the operator given by $(\flip\otimes\id)(\mathds{1}\otimes T)$. Similarly, if $t\in\mult{\mathnormal{B}}$, then $\mathds{1}\otimes t\otimes\mathds{1}\in\mult{\mathnormal{A}}{\mathnormal{B}}{\mathnormal{C}}$ will be denoted by $t_2$, and so on. 

For two vectors $\xi,\eta\in\mathcal{H}$, by the functional $\omega_{\xi,\eta}\in\B{H}_*$ we mean a map $T\mapsto\langle\xi|T\eta\rangle$ (note that the inner product is linear in the right variable).

We will use the following lemma a few times:
\begin{lemma}\label{lem:tensconv}
 Let $\mathnormal{A}\subseteq\B{H},\mathnormal{B}\subseteq\B{K}$ be concrete $C^*$-algebras and let $\mathsf{A}=\mathnormal{A}''$ and $\mathsf{B}=\mathnormal{B}''$. Let $T\in\mult{\mathnormal{A}}{\mathnormal{B}}\subseteq\mathsf{A}\vnt\mathsf{B}$ and let $(\omega_i)_{i\in I}\subseteq\B{H}_*$ be a net of normal functionals such that $\omega_i\xrightarrow{i\in I}\omega\in\B{H}^*$ in the weak$^*$-topology. Then $x_i=(\omega_i\otimes\id)T\xrightarrow{i\in I}(\omega\otimes\id)T=x$ in $\sigma$-WOT of $\mathsf{B}$.
\end{lemma}
\begin{proof}
 Recall that $\sigma$-WOT is the weak$^*$-topology in $\mathsf{B}$. Pick $\mu\in\mathsf{B}_*$.  We have
 \begin{equation}\label{eq:tensconv1}
 \mu(x_i)=(\omega_i\otimes\mu)(T)=\omega_i\bigl((\id\otimes\mu)T\bigr). 
 \end{equation}
Since $t=(\id\otimes\mu)T\in\mathsf{A}$ and $\omega_i\xrightarrow{i\in I}\omega\in\mathsf{A}^*$ weak$^*$, we have $\omega_i(t)\xrightarrow{i\in I}\omega(t)$. This is equivalent to
\[\mu(x_i)\xrightarrow{i\in I}\omega(t)=(\omega\otimes\mu)(T)=\mu\bigl((\omega\otimes\id)T\bigr)\]
which finishes the proof.
\end{proof}
\subsection{Locally compact quantum groups}
We work in the setting of locally compact quantum groups as defined by Kustermans and Vaes (\cite{KV00,Va01}), but sometimes we use the results proven in the setting of multiplicative unitaries (\cite{BS93,SoWo,SLW96}), as the latter axiomatization covers the former. Hence we use the following objects to study a quantum group $\GG$:
\begin{enumerate}
 \item The von Neumann algebra $\Linf{G}$, endowed with a coproduct $\Delta_{\GG}$ and n.s.f. weights $\varphi^{\GG}, \psi^{\GG}$ satisfying the left- and right-invariance conditions (called the left and right Haar weights, respectively);
 \item $\Lone{G}$, the predual of $\Linf{G}$;
 \item The reduced $C^{\ast}$-algebra $\Cnot{G}$, endowed with the same structure as above;
 \item The universal $C^{\ast}$-algebra $\Cunot{G}$, its coproduct will be denoted by $\Delta^u_{\GG}$ and the canonical quotient map will be denoted by $\Lambda_{\GG}\colon\Cunot{G}\to\Cnot{G}$;
 \item The Kac-Takesaki operator $\ww^{\GG}\in\mult{\Cnothat{G}}{\Cnot{G}}$ and its universal companions: $\Ww^{\GG}\in\mult{\Cunothat{G}}{\Cnot{G}}$, $\wW^{\GG}\in\mult{\Cnothat{G}}{\Cunot{G}}$ and $\WW^{\GG}\in\mult{\Cunothat{G}}{\Cunot{G}}$;
\end{enumerate}
The Kac-Takesaki operator $\ww^{\GG}$ implements the coproduct in $\Cnot{G}$ (and also in $\Linf{G}$) in the following way:
\begin{equation}\label{eq:coproduct}
\Delta_{\GG}(x)=\ww^{\GG}(x\otimes\mathds{1})(\ww^{\GG})^{\ast} 
\end{equation}
The semi-universal incarnations of the Kac-Takesaki operator are linked to the Kac-Takesaki operator by means of the reducing morphism:
\[(\id\otimes\Lambda_{\GG})\wW^{\GG}=\ww^{\GG}=(\Lambda_{\widehat{\GG}}\otimes\id)\Ww^{\GG}.\] 
The universal version of the Kac-Takesaki operator is linked to the semi-universal companions in a similar manner: 
\[(\id\otimes\Lambda_{\GG})\WW^{\GG}=\Ww^{\GG}\quad\text{and}\quad (\Lambda_{\widehat{\GG}}\otimes\id)\WW^{\GG}=\wW^{\GG}.\] 
These operators obey the following pentagonal-like equation (see \cite[Proposition 4.4]{MRW12}):
\begin{equation}\label{eq:universalpentagonal}
 \WW_{13}^{\GG}=(\Ww_{12}^{\GG})^*\wW_{23}^{\GG}\Ww_{12}^{\GG}(\wW_{23}^{\GG})^*\in\mult{\Cnothat{G}\otimes\mathnormal{K}(L^2(\GG))}{\Cnot{G}}
\end{equation}
The dual quantum group $\widehat{\GG}$ is governed by its own Kac-Takesaki operator $\ww^{\widehat{\GG}}=\flip(\ww^{\GG})^*$. Then the coproduct in $\Linfhat{G}$ takes the form:
\begin{equation}\label{eq:dualcoproduct}
\Delta_{\widehat{\GG}}(x)=\flip(\ww^{\GG})^{\ast}(x\otimes\mathds{1})\flip(\ww^{\GG})=\flip\left((\ww^{\GG})^{\ast}(\mathds{1}\otimes x)\ww^{\GG}\right). 
\end{equation}
When discussing a single locally compact quantum group, we will often drop the ${}^{\GG}$ and ${}_{\GG}$ decorations of the coproduct, Kac-Takesaki operators etc. Then the structure of the dual group will be decorated only with the hat decoration, e.g. $\widehat{\Delta}$ will be the coproduct in $\Linfhat{G}$ etc. The study of $\GG$ and $\widehat{\GG}$ is supplemented with
\begin{enumerate}\setcounter{enumi}{5}
 \item the unitary antipode $R\colon\Cnot{G}\to\Cnot{G}$, living also on the von Neumann algebra level: $R\colon\Linf{G}\to\Linf{G}$;
 \item its universal lift: $R^u\colon\Cunot{G}\to\Cunot{G}$;
 \item the scaling group: for every $t\in\mathbb{R}$ there is $\tau^{\GG}_t\colon\Cnot{G}\to\Cnot{G}$, extending to the von Neumann level: $\tau^{\GG}_t\colon\Linf{G}\to\Linf{G}$;
 \item the universal counterpart: the group of transformations $\tau^u_t\colon\Cunot{G}\to\Cunot{G}$.
 \item analogous structure for the dual: $\widehat{R}^u$, $\widehat{R}$, $\widehat{\tau}^u_t$ and $\widehat{\tau}_t$.
 \item the analytic continuation of $\tau_t$ to the upper half-plane $\mathbb{C}^+$ yields the analytic generator of the group of transformations $\tau_t$, denoted $\tau_{i/2}$. It appears in the polar decomposition of the antipode: $S=R\comp\tau_{i/2}$. In case $\tau_t=\id$, one has $\tau_{i/2}=\id$ and $S=R$ is bounded. In such cases we say that $\GG$ is of Kac type.
\end{enumerate}
The scaling groups and unitary antipodes are compatible with the reducing morphisms in the following sense:
\begin{equation}\label{eq:compatibility1}
\tau_t\comp\Lambda=\Lambda\comp\tau_t^{u}\qquad\mathrm{and}\qquad  R\comp\Lambda=\Lambda\comp R^{u}
\end{equation}
and similarly for $\widehat{\GG}$. Moreover, the scaling groups and unitary antipodes are compatible with the universal version of the Kac-Takesaki operator in the following sense (cf. \cite[Proposition 39, Lemma 40 and Proposition 42]{SoWo}):
\begin{equation}\label{eq:compatibility2}
(\widehat{\tau}_t^{u}\otimes\tau_t^{u})\WW=\WW\qquad\mathrm{and}\qquad  (\widehat{R}^{u}\otimes R^{u})\WW=\WW
\end{equation}

\subsection{Representation theory}\label{sec:LCQGrep}
By a \emph{representation} of a locally compact quantum group $\GG$ we will always mean a unitary element $U\in\mult{\K{H}{U}}{\Cnot{G}}$ satisfying:
\begin{equation}\label{eq:representation}
 (\id\otimes\Delta_{\GG})U=U_{12}U_{13}
\end{equation}
Recall that $\mult{\K{H}}{\Cnot{G}}\subseteq\B{H}{U}\vnt\Linf{G}$, we often just write $U\in\B{H}{U}\vnt\Linf{G}$. It is known that any $U\in\B{H}{U}\vnt\Linf{G}$ satisfying \eqref{eq:representation} actually satisfies $U\in\mult{\K{H}{U}}{\Cnot{G}}$, see e.g.~\cite[Theorem 1.6(2)]{SLW96}.

Out of two representations $U\in\mult{\K{H}{U}}{\Cnot{G}}$ and $\tilde{U}\in\mult{\K{H}{\tilde{U}}}{\Cnot{G}}$ one can form two new representations: the direct sum and tensor product, we will need a precise formula of the former later on, so we recall it. The direct sum is obtained by using the canonical inclusion maps $\iota\colon\B{H}{U}\hookrightarrow\mathsf{B}(\mathcal{H}_U\oplus\mathcal{H}_{\tilde{U}}))$ and $\tilde{\iota}\colon\B{H}{\tilde{U}}\hookrightarrow\mathsf{B}(\mathcal{H}_U\oplus\mathcal{H}_{\tilde{U}}))$ induced by the spatial maps $\mathcal{H}_U,\mathcal{H}_{\tilde{U}}\hookrightarrow\mathcal{H}_U\oplus\mathcal{H}_{\tilde{U}}$. Then the direct sum is nothing but \begin{equation}\label{eq:corepsum}U\oplus V\colon=(\iota\otimes\id)(U)+(\tilde{\iota}\otimes\id)\tilde{U}\in\mult{\mathnormal{K}(\mathcal{H}_U\oplus\mathcal{H}_{\tilde{U}})}{\Cnot{G}}.\end{equation}

The viewpoint $U\in\B{H}{U}\vnt\Linf{G}$ enables us to make sense of the following crucial observation, which is contained in \cite[Theorem 1.6]{SLW96}. Let $\omega\in\B{H}{U}_*$ be a normal functional. Then \begin{equation}\label{eq:antipodecorep}(\omega\otimes\id)U\in\mathcal{D}(S^{\GG})\quad\mathrm{and}\quad S^{\GG}\big((\omega\otimes\id)U\big)=(\omega\otimes\id)(U^*).\end{equation}

The representations of $\GG$ can always be realized by means of a ${}^*$-ho\-mo\-mor\-phism from $\Cunothat{G}$. We note the following result:
\begin{theorem}[{\cite[Proposition 5.2]{kustermans}}]\label{thm:kustermans}
 Let $U\in \mult{\K{H}{U}}{\Cnot{G}}$ be a representation. Then there exists a unique $\phi_U\in \Mor(\Cunothat{G},\K{H}{U})$ such that 
 \[(\phi_U\otimes\id)\Ww^{\GG}=U.\]
Conversely, given a $C^*$-algebra $\mathnormal{B}$, a representation in a Hilbert space $\rho\in \Mor(\mathnormal{B},\K{H})$ and morphism $\phi\in \Mor(\Cunothat{G}),\mathnormal{B})$, the unitary $U_{\phi,\rho}=(\rho\comp\phi\otimes\id)\Ww^{\GG}\in \mult{\K{H}}{\Cnot{G}}$ is a representation of $\GG$.
\end{theorem}
\subsection{Homomorphisms and subgroups}\label{sec:homosubgr}
An in-depth description of ho\-mo\-mor\-phisms between quantum groups was given in \cite{MRW12}, let us recall the main points. Fix two locally compact quantum groups $\mathbb{G}$ and $\mathbb{H}$. A ho\-mo\-mor\-phism of quantum groups $\mathbb{H}\to\mathbb{G}$ can be equivalently described by three objects:
\begin{description}[font=$\bullet\ $\scshape\bfseries]
 \item[Hopf ${}^*$-ho\-mo\-mor\-phisms] $\varphi\in \Mor(\Cunot{G},\Cunot{H})$, which intertwine the coproducts: 
 \[(\varphi\otimes\varphi)\comp\Delta_{\GG}^u=\Delta_{\HH}^u\circ\varphi\]
 \item[Bicharacters] unitary elements $V\in\mult{\Cnothat{G}}{\Cnot{H}}$, which are (anti)rep\-re\-sen\-ta\-tions on both legs:
 \begin{equation}\label{eq:bicharacter}
 (\Delta_{\widehat{\GG}}\otimes\id)V=V_{23}V_{13}\quad\mathrm{and}\quad(\id\otimes\Delta_{\HH})V=V_{12}V_{13} 
 \end{equation} Moreover, they satisfy $(\widehat{R}^{\GG}\otimes R^{\HH})V=V$ and $(\tau^{\widehat{\GG}}_t\otimes \tau^{\HH}_t)V=V$. 
 \item[Right quantum group ho\-mo\-mor\-phisms] morphisms $\rho\in \Mor(\Cnot{G},\Cnot{G}\otimes\Cnot{H})$ satisfying
 \begin{equation}\label{eq:rightquantumgroupmorphism}
(\Delta_{\GG}\otimes\id)\circ\rho=(\id\otimes\rho)\circ\Delta_{\GG}\quad\mathrm{and}\quad(\id\otimes\Delta_{\HH})\circ\rho=(\rho\otimes\id)\circ\rho   
 \end{equation} Moreover, they satisfy $(\id\otimes\rho)\ww^{\GG}=\ww^{\GG}_{12}V_{13}$, where $V$ is the corresponding bicharacter.
\end{description}
Furthermore, each ho\-mo\-mor\-phism $\HH\to\GG$ has its dual ho\-mo\-mor\-phism $\widehat{\GG}\to\widehat{\HH}$. It can be described as follows. If $\varphi\in\Mor(\Cunot{G},\Cunot{H})$ is a Hopf ${}^*$-ho\-mo\-mor\-phism, then there exists a unique $\widehat{\varphi}\in\Mor(\Cunothat{H},\Cunothat{G})$, these maps are linked via \begin{equation}\label{eq:dualhomo}(\id\otimes\varphi)\WW^{\HH}=(\widehat{\varphi}\otimes\id)\WW^{\GG}.\end{equation} Equivalently, if $V\in\mult{\Cnothat{G}}{\Cnot{H}}$ is a bicharacter representing a ho\-mo\-mor\-phism $\mathbb{H}\to\mathbb{G}$, then $\widehat{V}=\flip(V^*)\in\mult{\Cnot{H}}{\Cnothat{G}}$ is a bicharacter representing its dual ho\-mo\-mor\-phism $\widehat{\GG}\to\widehat{\HH}$. 

Let us also stress that bicharacters and right quantum group homomorphisms are equally well studied in the von Neumann algebraic context, so that a unitary $V\in\Linfhat{G}\vnt\Linf{H}$ satisfying \eqref{eq:bicharacter} and a normal ${}^*$-ho\-mo\-mor\-phism $\rho\colon\Linf{G}\to\Linf{G}\vnt\Linf{H}$ satisfying \eqref{eq:rightquantumgroupmorphism} also describe a ho\-mo\-mor\-phism of quantum groups $\HH\to\GG$. Right quantum group homomorphisms in the von Neumann algebraic context are in fact normal extensions of the respective maps in the $C^*$-algebraic context: they are implemented by $V$ by the formula $\rho(x)=V(x\otimes\mathds{1})V^*$. Let us note that second condition in \eqref{eq:rightquantumgroupmorphism} corresponds to $\rho$ being a right action of $\HH$ on $\Linf{G}$. These are analogues of the natural actions by right shifts.

A thorough treatment of the notion of subgroup was given in \cite{DKSS}, we recall some of the main points of that article. Let $\HH\to\GG$ be a ho\-mo\-mor\-phism of quantum groups (described by a Hopf ${}^*$-ho\-mo\-mor\-phism $\varphi\in \Mor(\Cunot{H},\Cunot{G})$, a bicharacter $V\in\mult{\Cnothat{G}}{\Cnot{H}}$ and a right quantum group ho\-mo\-mor\-phism $\rho\in \Mor(\Cnot{G},\Cnot{G}\otimes\Cnot{H})$). We say that $\HH$ is a closed quantum subgroup\footnote{or sometimes we call them Vaes-closed quantum subgroups, as there is a competing definition of Woronowicz-closed quantum subgroup and they agree in case $\GG$ is compact, discrete, classical or dual to classical, see \cite{DKSS,KKSint}} of $\GG$ if there exists a normal injective ${}^*$-ho\-mo\-mor\-phism $\gamma\colon\Linfhat{H}\to\Linfhat{G}$ such that $V=(\gamma\otimes\id)\ww^{\HH}$. This map $\gamma$ is nothing but the incarnation on the von Neumann algebra level of the reduced version of $\hat{\pi}$, namely
\begin{equation}\label{eq:reduceddualhomo1}
\Lambda_{\GG}\comp\widehat{\varphi}=\gamma\comp\Lambda_{\HH} 
\end{equation}
and hence, by \eqref{eq:dualhomo}, in particular we have
\begin{equation}\label{eq:reduceddualhomo2}
(\id\otimes\varphi)\wW^{\GG}=(\gamma\otimes\id)\wW^{\HH} 
\end{equation}
Let $\mathsf{M}\subseteq\Linf{G}$ be a von Neumann subalgebra. Recall from \cite{TT72} that $\mathsf{M}$ is called \emph{invariant} if $\Delta(\mathsf{M})\subseteq\mathsf{M}\vnt\mathsf{M}$. We moreover say that $\mathsf{M}$ is a \emph{Baaj-Vaes subalgebra} if it is invariant, $R(\mathsf{M})=\mathsf{M}$ and $\tau_t(\mathsf{M})=\mathsf{M}$ for all $t\in\mathbb{R}$. The Baaj-Vaes theorem \cite[Proposition A.5]{BV05} states that Baaj-Vaes subalgebras of $\Linf{G}$ are in one to one correspondence with closed quantum subgroups of $\widehat{\GG}$. This means, in particular, that $\mathsf{M}$ can be endowed with Haar weights, and that the inclusion $\iota\colon\mathsf{M}\hookrightarrow\Linf{G}$ induces the full data of a quantum group ho\-mo\-mor\-phism: the morphism $\pi\in\Mor(\Cunothat{G},\Cunothat{H})$, the bicharacter and the right quantum group ho\-mo\-mor\-phism. Moreover, $\iota$ is the von Neumann incarnation of the reduced version of $\hat{\pi}$.

Let us comment on the assumptions of Baaj-Vaes theorem. It turns out that in the case $\GG$ is compact, discrete, classical and dual to classical the conditions $R(\mathsf{M})=\mathsf{M}$ and $\tau_t(\mathsf{M})=\mathsf{M}$ for all $t\in\mathbb{R}$ actually follow from some more general principles. In case $\GG$ is compact, this follows from restriction of Haar state to $\mathsf{M}$, in case $\GG$ is discrete this is \cite[Theorem 3.1]{NY14}, after applying co-duality techniques of \cite{KS14a} (an elementary proof is also available, see \cite[Theorem 1.35]{PJphd}). The case of classical groups and dual to classical groups was covered in \cite{TT72} by Takesaki and Tatsuuma. It turns out that the von Neumann algebras constructed in \autoref{sec:hopfimageconstruction} are automatically invariant, and once made $\tau_t$-invariant for all $t\in\mathbb{R}$, they are also $R$-invariant. This covers an abundance of invariant subalgebras, especially in the Kac case, and it is not known to the authors whether there exists an invariant von Neumann subalgebra that is not a Baaj-Vaes subalgebra.
\section{Construction of Hopf image}\label{sec:hopfimageconstruction}
\subsection{First steps towards the construction}\label{sec:construction}
The goal of this part is to construct a quantum group $\HH$ which will later be shown to satisfy the defining properties of Hopf image. So let us fix a morphism $\beta\in \Mor(\Cunot{G},\mathnormal{B})$, where $\mathnormal{B}$ is some $C^{\ast}$-algebra, as in the introduction.

Application of $\Lambda_{\widehat{\mathbb{G}}}\otimes\id\otimes\id$ to both sides of \eqref{eq:universalpentagonal} yields:
\begin{equation}\label{eq:semireducedcharacter4}
\wW_{23}\ww_{12}\wW_{23}^{\ast}=\ww_{12}\wW_{13}. 
\end{equation}
Let us denote $X=(\id\otimes\beta)\wW\in\mult{\Cnothat{G}}{\mathnormal{B}}$. Computing the value of $(\id\otimes\id\otimes\beta)$ at both sides of the equality \eqref{eq:semireducedcharacter4} results in:
\begin{equation}\label{eq:dualcorepresentation1}
X_{23}\ww_{12}X_{23}^{\ast}=\ww_{12}X_{13} 
\end{equation}
or, equivalently, 
\begin{equation}\label{eq:dualcorepresentation2}
(\Delta_{\widehat{\mathbb{G}}}\otimes\id)X=X_{23}X_{13}
\end{equation}
Applying $(\omega\otimes\id\otimes\id)$ for $\omega\in\Lonehat{G}$ to both sides of \eqref{eq:dualcorepresentation1} we obtain:
\begin{equation}\label{eq:partialcoaction0}
X(a\otimes\mathds{1})X^{\ast}=(\omega\otimes\id\otimes\id)(\ww_{12}X_{13}), 
\end{equation}
where $a=(\omega\otimes\id)\ww$. As
\begin{equation}\label{eq:comparisonofmultiplieralgebras}\begin{split}
\ww_{12}\in \mult{\Cnothat{G}}{\Cnot{G}}{\mathnormal{B}}\subseteq \mult{\mathnormal{K}(L^2(\GG))}{\Cnot{G}}{\mathnormal{B}}\\
X_{13}\in \mult{\Cnothat{G}}{\Cnot{G}}{\mathnormal{B}}\subseteq \mult{\mathnormal{K}(L^2(\GG))}{\Cnot{G}}{\mathnormal{B}}
\end{split}
\end{equation}
we get that $X(a\otimes\mathds{1})X^{\ast}\in \mult{\Cnot{G}}{\mathnormal{B}}$ and hence the map $\theta\colon\Cnot{G}\to\mult{\Cnot{G}}{\mathnormal{B}}$ can be defined by
\begin{equation}\label{eq:partialcoaction1}
\Cnot{G}\ni a\xmapsto{\ \theta\ } X(a\otimes\mathds{1})X^{\ast}\in \mult{\Cnot{G}}{\mathnormal{B}}. 
\end{equation}
Let us assume that $\mathnormal{B}$ is (faithfully, nondegenerately) represented on a Hilbert space $\mathcal{H}$: $\mathnormal{B}\subseteq\B{H}$. Then we can view $\theta$ as a representation $\theta\in \Rep(\Cnot{G},L^2(\GG)\otimes\mathcal{H})$. One has:
\begin{equation}\label{eq:partialcoaction2}
(\id\otimes\,\theta)(\ww)=\ww_{12}X_{13}\in\mult{\Cnothat{G}}{\Cnot{G}}{\mathnormal{B}}\subseteq \mult{\Cnothat{G}}{\mathnormal{K}(L^2(\mathbb{G})\otimes\mathcal{H})}.\end{equation}
As $\ww\in\mult{\Cnothat{G}}{\Cnot{G}}$ generates $\Cnot{G}$ in the sense of \cite{SLW95a}, we conclude that:
\begin{proposition}
$\theta\in \Mor(\Cnot{G},\Cnot{G}\otimes \mathnormal{B})$.  
\end{proposition}

Now we are in position to state the main construction. Let 
\[\mathcal{M}=\left\{(\id\otimes\omega)X \mid \omega\in \mathnormal{B}^{\ast}\right\}\subseteq\Linfhat{G}\]
As $X\in\mult{\Cnothat{G}}{\mathnormal{B}}$ one sees that $\mathcal{M}\subseteq \mult{\Cnothat{G}}\subseteq\Linfhat{G}$. Denote by $\mathcal{M}_1$ the ${}^{\ast}$-algebra generated by $\mathcal{M}$ and by $\mathnormal{M}_1$ its norm-closure. 

\begin{proposition}\label{prop:mzeroprimvna}
$\mathcal{M}_1'=\mathnormal{M}_1'$ is a von Neumann algebra
\end{proposition}
\begin{proof}
Indeed, as we have  
\[T\in \mathcal{M}_1'\iff X(T\otimes\mathds{1})=(T\otimes\mathds{1})X\iff (T\otimes\mathds{1})X^{\ast}=X^{\ast}(T\otimes\mathds{1})\iff\] \[\iff X(T^{\ast}\otimes\mathds{1})=(T^{\ast}\otimes\mathds{1})X\iff T^{\ast}\in \mathcal{M}_1'\]
\end{proof}
\noindent Thus also $\mathsf{M}_1=\mathcal{M}''=\mathcal{M}_1^{\wc}$ is a von Neumann algebra.

Let $\mathsf{M}_{BV}$ be the smallest Baaj-Vaes subalgebra containing $\mathnormal{M}_1$ (so in particular containing $\mathsf{M}_1$). The existence of such a von Neumann algebra follows from standard argument: it is the intersection of all Baaj-Vaes subalgebras of $\Linfhat{G}$ containing $\mathsf{M}_1$: this collection is non-empty because $\Linfhat{G}$ itself is such an algebra. Later on we will see that it can be constructed more explicitly. Thanks to Baaj-Vaes theorem, there exists $\mathbb{H}\subset\mathbb{G}$ such that $\Linfhat{H}=\mathsf{M}_{BV}$, in particular, we have a map $\pi\in\Mor(\Cunot{G},\Cunot{H})$ coming from \cite[Theorem 3.5]{DKSS}, which is linked to the embedding $\mathsf{M}_{BV}\subseteq\Linfhat{G}$ via \eqref{eq:reduceddualhomo2}. 

\subsection{Properties of the algebra \texorpdfstring{$\mathsf{M}_1$}{M1}}\label{sec:hopfimagetechnical}
\begin{lemma}\label{lem:descriptionofM}
Let $\beta\in\Mor(\Cunot{G},\mathnormal{B})$ and assume $\mathnormal{B}$ is faithfully, nondegenerately represented on a Hilbert space $\mathcal{H}$: $\mathnormal{B}\subseteq\B{H}$. Let $\mathnormal{C}=\beta(\Cunot{G})\subseteq\mult{\mathnormal{B}}$. Denote:
 \begin{itemize}
  \item $\mathsf{M}_1=\left\{(\id\otimes\omega)(X) \mid \omega\in \mathnormal{B}^{\ast}\right\}''$
  \item $\mathsf{M}_2=\left\{(\id\otimes\omega)(X) \mid \omega\in \mathnormal{C}^{\ast}\right\}''$
  \item $\mathsf{M}_3=\left\{(\id\otimes\omega)(X) \mid \omega\in\B{H}^*\right\}''$
  \item $\mathsf{M}_4=\left\{(\id\otimes\omega)(X) \mid \omega\in\B{H}_*\right\}''$
  \item $\mathsf{M}_5=\left\{(\id\otimes\omega_{\xi,\eta})(X) \mid \xi,\eta\in\mathcal{H}\right\}''$
 \end{itemize}
Then
\begin{itemize}
 \item $\beta\in\Mor(\Cunot{G},\mathnormal{C})$
 \item $\mathsf{M}_1=\mathsf{M}_2=\mathsf{M}_3=\mathsf{M}_4=\mathsf{M}_5$
\end{itemize}
\end{lemma}
\begin{remark}\label{rmk:embedding}
By first part of \autoref{lem:descriptionofM} we see that we can restrict our attention to maps $\beta\colon\Cunot{G}\to \mathnormal{B}$ that are surjective (philosophically speaking, the maps that are Gelfand duals of embeddings $\widehat{\beta}\colon\mathbb{X}\hookrightarrow\GG$ of a quantum space $\mathbb{X}$ as a closed quantum subset of $\GG$, where $\mathnormal{B}=\Cnot{X}$).  
\end{remark}

\begin{proof}
 The first statement follows from standard reasoning, so we omit it.

It is obvious that $\mathsf{M}_5\subseteq\mathsf{M}_4\subseteq \mathsf{M}_3$. To show that $\mathsf{M}_3\subseteq \mathsf{M}_4$, let us fix an element $\omega\in\B{H}^*$ and consider $x=(\id\otimes\omega)X$. Recall that $\B{H}_*\subseteq\B{H}^*$ is weak$^*$-dense, so pick a net $(\omega_i)_{i\in I}\subseteq\B{H}_*$ such that $\omega_n\xrightarrow{i\in I}\omega$ in weak$^*$-topology. Then $x_i=(\id\otimes\omega_i)X\in\mathsf{M}_4$ and by \autoref{lem:tensconv} we have that $x_i\to x$ in $\sigma$-WOT. Hence $x\in\mathsf{M}_3$ by $\sigma$-WOT closedness of the latter and we are done.

Now as every functional in $\B{H}^*$ restricts to $\mathnormal{B}$ and $\mathnormal{C}$ we have that $\mathsf{M}_3\subseteq \mathsf{M}_1, \mathsf{M}_2$. But as $\mathnormal{B}, \mathnormal{C}\subseteq\B{H}$ are closed, any continuous functional from $\mathnormal{B}^{\ast}$ and $\mathnormal{C}^{\ast}$ extends to a continuous functional in $\B{H}^*$ by Hahn-Banach theorem, so $\B{H}^*\twoheadrightarrow \mathnormal{C}^{\ast}, \mathnormal{B}^{\ast}$. In particular, this means $\mathsf{M}_1, \mathsf{M}_2\subseteq \mathsf{M}_3$.

For the equality $\mathsf{M}_4=\mathsf{M}_5$ recall that the linear span of vector functionals $\omega_{\xi,\eta}$ is norm dense in $\B{H}_*$ (see, e.g. \cite[III.2.1.4]{encyclopaedia}). Now by standard calculation we show that if $\omega_i\xrightarrow{i\in I}\omega$ in norm, then $(\id\otimes\omega_i)X\xrightarrow{i\in I}(\id\otimes\omega)X$ in WOT. Pick then $\xi,\eta\in\mathcal{H}$, we have:
\[\begin{split}\left|\langle\xi|(\id\otimes(\omega-\omega_n))X|\eta\rangle\right|=|(\omega_{\xi,\eta}\otimes(\omega-\omega_n))X|\\
\leq\|\omega_{\xi,\eta}\otimes(\omega-\omega_n)\|\|X\|=|\langle\xi|\eta\rangle|\|\omega-\omega_n\|\|X\|\to0.
  \end{split}\]
By WOT-closedness of $\mathsf{M}_5$ any element of the generating set of $\mathsf{M}_4$ is in fact in $\mathsf{M}_5$, so we conclude by von Neumann's bicommutant Theorem.
\end{proof}
\begin{proposition}\label{prop:Misinvariant}
The algebra $\mathsf{M}_1$ is invariant, i.e.~$\widehat{\Delta}(\mathsf{M}_1)\subseteq \mathsf{M}_1\vnt \mathsf{M}_1$. If moreover $\tau_t^u(\ker\beta)\subseteq\ker\beta$ (in particular if $\mathbb{G}$ is Kac type), then $\mathsf{M}_1$ is is preserved by $\widehat{\tau}_{t}$ for each individual $t\in\mathbb{R}$. 
\end{proposition}
\begin{proof}
 For the invariance, let us first pick $x=(\id\otimes\omega_{\xi,\eta})X\in \mathsf{M}_5$ for some $\xi,\eta\in\mathcal{H}$, we will show that $\widehat{\Delta}(x)\in \mathsf{M}_5\vnt \mathsf{M}_5$. Pick an orthonormal basis $(e_j)_{j\in J}$ of $\mathcal{H}$ and recall that $\mathds{1}=\sum_{j\in J}|e_j\rangle\langle e_j|$ is a WOT-convergent resolution of identity into rank one projections. We compute: 
 \[\begin{split}
    \widehat{\Delta}(x)&=(\widehat{\Delta}\otimes\omega_{\xi,\eta})(X)=(\id\otimes \id\otimes\omega_{\xi,\eta})(X_{23}X_{13})=\\
    &=(\id\otimes \id\otimes\omega_{\xi,\eta})(X_{23}(\mathds{1}\otimes \mathds{1}\otimes\sum_{j\in J}|e_j\rangle\langle e_j|)X_{13})=\\
    &=\sum_{j\in J}(\id\otimes \id\otimes\omega_{\xi,\eta})(X_{23}(\mathds{1}\otimes \mathds{1}\otimes|e_j\rangle\langle e_j|)X_{13})=\\
    &=\sum_{j\in J}(\id\otimes \id\otimes\omega_{\xi,e_j})(X_{23})(\id\otimes \id\otimes\omega_{e_j,\eta}\rangle)(X_{13})=\\
    &=\sum_{j\in J}(\id\otimes\omega_{\xi,e_j})(X)\otimes (\id\otimes\omega_{e_j,\eta})(X)\in \mathsf{M}_5\vnt\mathsf{M}_5\\
   \end{split}\]
 We conclude by normality of $\widehat{\Delta}$ and equality $\mathsf{M}_1=\mathsf{M}_5$ obtained in \autoref{lem:descriptionofM}.
 
 Let $t\in\mathbb{R}$, assume $\ker(\beta)$ is $\tau_t^u$ invariant and let $\omega\in \mathnormal{B}^*$. Then there exists a (necessarily unique) functional $\omega_t\in \mathnormal{B}^*$ such that $\omega\comp\beta\comp\tau_t^u=\omega_t\comp\beta$. This follows from a general Banach space theory: there exists an isometry $s$ making the following diagram commute:
 
 \begin{center}
\begin{tikzpicture}
  [bend angle=36,scale=2,auto,
pre/.style={<-,shorten <=1pt,semithick},
post/.style={->,shorten >=1pt,semithick}]
\node (G) at (-1,1) {$\Cunot{G}$};
\node (B) at (1,1) {$\mathnormal{C}$}
edge [pre] node[auto,swap] {$\beta$} (G);
\node (Q) at (0,0) {$\bigslant{\Cunot{G}}{\ker\beta}$}
edge [pre] node[auto,swap] {$q$} (G)
edge [post] node[auto] {$s$} (B);
 \end{tikzpicture}  
 \end{center}

Now, as $\tau_t(\ker\beta)\subseteq\ker\beta$, using $s$ we can conclude the existence of a map $\tau^{\mathnormal{C}}_t\colon\mathnormal{C}\to\mathnormal{C}$ such that $\beta\comp\tau_t^u=\tau_t^{\mathnormal{C}}\comp\beta$. Then

\[\begin{split}
 \{(\id\otimes\omega\comp\beta)\wW \mid \omega\in\mathnormal{C}^*\}''=&\{(\id\otimes\omega\comp\beta)(\hat{\tau}_t\otimes\tau^u_t)\wW \mid \omega\in\mathnormal{C}^*\}''=\\
 =&\hat{\tau}_t\biggl(\{(\id\otimes\omega\comp\beta\comp\tau^u_t)\wW \mid \omega\in\mathnormal{C}^*\}''\biggr)=\\
 =&\hat{\tau}_t\biggl(\{(\id\otimes\omega\comp\tau^C_t\comp\beta)\wW \mid \omega\in\mathnormal{C}^*\}''\biggr)=\\
 =&\hat{\tau}_t\biggl(\{(\id\otimes\omega\comp\beta)\wW \mid \omega\in\mathnormal{C}^*\}''\biggr)
\end{split}\]

Because $\tau_t^{\mathnormal{C}}\colon\mathnormal{C}\to\mathnormal{C}$ is a bijection, also $(\tau_t^{\mathnormal{C}})^*\colon\mathnormal{C}^*\to\mathnormal{C}^*$ is a bijection.
\end{proof}

\begin{theorem}\label{thm:BVclosure}
 The minimal Baaj-Vaes subalgebra of $\Linfhat{G}$ containing $\mathsf{M}_1$ is given by \[\mathsf{M}_{BV}=(\bigcup_{t\in\mathbb{R}}\widehat{\tau}_t(\mathsf{M}_1))''.\] In particular, $\mathsf{M}_{BV}=\mathsf{M}_1$ if $\mathbb{G}$ is compact or discrete or if $\tau_t^u(\ker(\beta))\subseteq\ker(\beta)$ for all $t\in\mathbb{R}$.
\end{theorem}
 \begin{proof}
For the purpose of the proof, let us denote by $\mathsf{M}_{\text{\text{RHS}}}$ the von Neumann algebra appearing on the right hand side of \autoref{thm:BVclosure}. This algebra is clearly $\widehat{\tau}$-invariant, and as $\mathsf{M}_1$ is invariant, using \cite[Theorem 1.9(2)]{MNW03} we conclude that $\mathsf{M}_{\text{\text{RHS}}}$ is again invariant. To see that $\mathsf{M}_{\text{\text{RHS}}}=\mathsf{M}_{BV}$, we need to show that it is preserved by the unitary antipode $\widehat{R}$.

First, for $t\in\mathbb{R}$ and $\omega\in\B{H}_*$ let us denote $x_{\omega,t}=(\widehat{\tau}_t\otimes\omega)X$ and observe that $\mathsf{M}_1=\mathsf{M}_4$ is generated by $x_{\omega,0}$ for all $\omega\in\B{H}_*$. Furthermore, for $t\in\mathbb{R}$ fixed, $\widehat{\tau}_t(\mathsf{M}_1)$ is generated by $\{x_{\omega,t} \mid \omega\in\B{H}_*\}$. In turn, $\mathsf{M}_{\text{\text{RHS}}}$ is generated by $\{ x_{\omega,t} \mid \omega\in\B{H}_*, t\in\mathbb{R}\}$. Indeed, from the above description we see that all elements $x_{\omega,t}\in\mathsf{M}_{\text{\text{RHS}}}$. The converse inclusion follows easily from von Neumann's bicommutant Theorem: if $y$ commutes with all $x_{\omega,t}$ for all $\omega$ and $t$, then $y\in\bigcap_{t\in\mathbb{R}}\widehat{\tau}_t(\mathsf{M}_1)'\subseteq (\bigcup_{t\in\mathbb{R}}\widehat{\tau}_t(\mathsf{M}_1))'$.

Now we use $\mathsf{M}_1=\mathsf{M}_4$ from \autoref{lem:descriptionofM}. Let us pick $\omega\in\B{H}_*$. Then $\omega^*$ defined as $\omega^*(a)=\overline{\omega(a^*)}$ is again a normal functional. Then by \eqref{eq:antipodecorep} we have that $(\id\otimes\omega)(X^*)=[(\id\otimes\omega^*)(X)]^*=x_{\omega^*,0}^*\in\mathcal{D}(\widehat{S})=\mathcal{D}(\widehat{\tau}_{i/2})$ (recall that in fact $X$ is antirepresentation, so $\flip(X^*)$ is a representation of $\mathbb{G}$), see \cite{SLW96}. Hence
\begin{equation}\label{eq:Sinvariance}
 \widehat{R}\bigl((\id\otimes\omega)(X^*)\bigr)=(\widehat{\tau}_{-i/2}\comp\widehat{S})\bigl((\id\otimes\omega)X^*\bigr)=\widehat{\tau}_{-i/2}\bigl((\id\otimes\omega)X\bigr)
\end{equation}
But as $\mathsf{M}_{\text{\text{RHS}}}$ is $\widehat{\tau}$-invariant, it is also preserved by its analytic generator, as it is defined uniquely: the analytic generator of $\widehat{\tau}_{-t}\restriction_{\mathsf{M}_{\text{RHS}}}$ is precisely $(\widehat{\tau}_{-i/2})\restriction_{\mathsf{M}_{\text{RHS}}}$ (this is in fact the uniqueness of analytic continuation of a function). Hence $\widehat{R}(x_{\omega^*,0}^*)\in\mathsf{M}_{\text{RHS}}$ for all $\omega\in\B{H}_*$ and in turn $\widehat{R}(x_{\omega,0})\in\mathsf{M}_{\text{RHS}}$ for all $\omega\in\B{H}_*$. Next, thanks to the relation $\widehat{\tau}_t\comp\widehat{R}=\widehat{R}\comp\widehat{\tau}_t$, it follows that $\widehat{R}(x_{\omega,t})=\widehat{\tau}_t(\widehat{R}(x_{\omega,0}))\in\widehat{\tau}_t(\mathsf{M}_{\text{RHS}})=\mathsf{M}_{\text{RHS}}$. Together with the description of the generating set in the first step, this finishes the proof of the main assertion.

The last part of the Theorem follows from the observation that in all these cases the algebra $\mathsf{M}_1$ is automatically $\widehat{\tau}$-invariant: in case of $\mathbb{G}$ compact or discrete this was discussed at the end of \autoref{sec:homosubgr} and the case $\tau_t^u(\ker(\beta))\subseteq\ker(\beta)$ for all $t\in\mathbb{R}$ follows from the discussion in \autoref{prop:Misinvariant}. 
\end{proof}

\subsection{Verification of defining properties}\label{sec:verification}
In this part we show that the quantum subgroup $\HH$ constructed in \autoref{sec:construction} indeed satisfies the defining properties of the Hopf image, i.e.~firstly, there exists $\tilde{\beta}\in \Mor(\Cunot{H},\mathnormal{B})$ as described in \bref{diagram}{Diagram 1}, that is showing that $(\pi,\mathbb{H},\tilde{\beta})\in\mathcal{C}_{\beta}$, and secondly, that it is an initial object of the category $\mathcal{C}_{\beta}$.
\begin{lemma}\label{lem:minimality}
 Let $\KK$ be a closed quantum subgroup of $\GG$ and denote the associated Hopf ${}^*$-ho\-mo\-mor\-phism and normal embedding by $\pi_{\KK}\in \Mor(\Cunot{G},\Cunot{K}))$ and $\gamma\colon\Linfhat{K}\to\Linfhat{G}$, respectively. Then $(\pi_{\KK}, \KK, \tilde{\beta})\in\mathcal{C}_{\beta}$, i.e.~there exists $\tilde{\beta}\in \Mor(\Cunot{K},\mathnormal{B})$ such that $\beta=\tilde{\beta}\comp\pi_{\KK}$ if and only if $\mathnormal{M}_1\subseteq \gamma(\Linfhat{K})$.
\end{lemma}
\begin{proof}
 Assume that $(\pi_{\KK},\KK,\tilde{\beta})\in\mathcal{C}_{\beta}$ and pick $\omega\in \mathnormal{B}^{\ast}$. Using \eqref{eq:reduceddualhomo1} and \eqref{eq:reduceddualhomo2}, we have that 
 \[(\id\otimes\omega)(\id\otimes\beta)(\wW^{\mathbb{G}})=(\id\otimes\omega)(\id\otimes\tilde{\beta}\comp\pi_{\KK})(\wW^{\mathbb{G}})=\gamma\bigg((\id\otimes\omega\comp\tilde{\beta})\wW^{\mathbb{K}}\bigg)\]
 
 Hence the algebra $\mathcal{M}_1$ constructed for $\beta$ and the corresponding algebra constructed for $\tilde{\beta}$, seen as subalgebras of $\Linfhat{G}$, coincide, hence so do their $C^{\ast}$-envelopes $\mathnormal{M}_1$ and the corresponding one for $\tilde{\beta}$. This shows the necessity.
 
 Assume that $\mathnormal{M}_1\subseteq\gamma(\Linfhat{K})$. In order to get a morphism $\tilde{\beta}\in \Mor(\Cunot{K},\mathnormal{B})$ it is enough to show that $\flip(X^*)$ is a representation of $\widehat{\mathbb{K}}$ by \autoref{thm:kustermans}.
 
 Now observe that \cite[Lemma 1.4]{DKSS} shows that $X\in\mult{\mathnormal{M}_1}{\mathnormal{B}}\subseteq\mult{\Cnothat{H}}{\mathnormal{B}}$. But as $\Delta_{\widehat{\mathbb{K}}}$ and $\Delta_{\widehat{\mathbb{G}}}\restriction_{\gamma(\Linfhat{K})}$ coincide and from \eqref{eq:dualcorepresentation2}, we have that $X_{23}X_{13}=(\Delta_{\widehat{\mathbb{K}}}\otimes \id)X$, so $\flip(X^*)$ satisfies hypothesis of \autoref{thm:kustermans}.
\end{proof}
\begin{proof}[Proof of \autoref{thmintro:existence}]
From \autoref{lem:minimality} we get that $\mathbb{H}$ constructed at the end of \autoref{sec:construction} can be endowed with the morphism $\tilde{\beta}$ completing the desired factorization, i.e.~$(\pi,\mathbb{H},\tilde{\beta})\in\mathcal{C}_{\beta}$. Let now $\mathds{k}=(\pi_{\mathbb{K}},\mathbb{K},\beta')\in\mathcal{C}_{\beta}$. From \autoref{lem:minimality} we have that $M_0\subseteq\Linfhat{K}$ and $\Linfhat{K}$ is a $R^{\widehat{\GG}}$-,$\tau_t^{\widehat{\GG}}$- and $\Delta_{\widehat{\GG}}$-invariant subalgebra of $\Linfhat{G}$. As $\Linfhat{H}$ is chosen to be a minimal von Neumann subalgebra with this property, we necessarily have $\Linfhat{H}\subseteq\Linfhat{K}$. In particular the inclusion map satisfies the defining property of $\HH$ being a closed quantum subgroup of $\KK$, so we conclude by \cite[Theorem 3.5]{DKSS} and \cite[Lemma 2.5]{KSS}.
\end{proof}
\subsection{More on \texorpdfstring{$X\in\mult{\Cnothat{G}}{\mathnormal{B}}$}{X} and \texorpdfstring{$\theta\in \Mor(\Cnot{G},\Cnot{G}\otimes \mathnormal{B})$}{theta}}\label{sec:MRWlike}
In this section we investigate the mutual relation between the objects describing the \emph{embedding} $\mathbb{X}\hookrightarrow\GG$ as phrased in \autoref{rmk:embedding}, i.e.~the morphism $\beta\in \Mor(\Cunot{G},\mathnormal{B})$, the unitary antirepresentation $X\in\mult{\Cnothat{G}}{\mathnormal{B}}$ and the morphism $\theta\in \Mor(\Cnot{G},\Cnot{G}\otimes \mathnormal{B})$ in the spirit of \cite{MRW12}. 

From the discussion in \autoref{sec:construction} it is clear that out of $\beta$ one can canonically construct the unitary $X$, which is an antirepresentation of $\widehat{\GG}$. But \autoref{thm:kustermans} (applied to $\flip(X^*)$ as in the proof of \autoref{lem:minimality}) shows that to a unitary $X\in\mult{\Cnothat{G}}{\mathnormal{B}}$ there corresponds a unique morphism $\beta\in \Mor(\Cunot{G},\mathnormal{B})$. 

Again, from the discussion in \autoref{sec:construction} it is clear that out of a unitary $X$, which is an antirepresentation of $\widehat{\GG}$ one can uniquely construct the morphism $\theta\in \Mor(\Cnot{G},\Cnot{G}\otimes\mathnormal{B})$. Observe that this morphism satisfies the following condition: $(\Delta\otimes\id)\comp\theta=(\id\otimes\,\theta)\comp\Delta$. Indeed, for $a\in\Cnot{G}$ we have that
\begin{equation}\label{eq:rightaction1}(\Delta\otimes \id)\comp\theta(a)=(\Delta\otimes \id)(X(a\otimes\mathds{1})X^{\ast})=\ww_{12}X_{13}(a\otimes\mathds{1}\otimes\mathds{1})X_{13}^{\ast}\ww_{12}^{\ast}.                                      
\end{equation}
Now using \eqref{eq:dualcorepresentation1}, we can continue calculations from \eqref{eq:rightaction1} and get:
\begin{equation}\label{eq:rightaction2}
 \begin{split}
\ww_{12}X_{13}(a\otimes\mathds{1}\otimes\mathds{1})X_{13}^{\ast}\ww_{12}^{\ast}=&X_{23}\ww_{12}X_{23}^{\ast}(a\otimes\mathds{1}\otimes\mathds{1})X_{23}\ww_{12}^{\ast}X_{23}^{\ast}=\\
=&X_{23}\ww_{12}(a\otimes\mathds{1}\otimes\mathds{1})\ww_{12}^{\ast}X_{23}^{\ast}=\\
=&X_{23}\bigg(\ww(a\otimes\mathds{1})\ww^{\ast}\bigg)_{12}X_{23}=\\
=&X_{23}(\Delta(a))_{12}X_{23}^{\ast}=(\id\otimes\,\theta)\comp\Delta(a)
 \end{split}
\end{equation}
\begin{proposition}\label{prop:MRWlike}
 Assume $\theta\in \Mor(\Cnot{G},\Cnot{G}\otimes \mathnormal{B})$ is such that $(\Delta\otimes \id)\comp\theta=(\id\otimes\,\theta)\comp\Delta$. Then there is a unique unitary $X\in\mult{\Cnothat{G}}{\mathnormal{B}}$ such that $\theta(a)=X(a\otimes\mathds{1})X^{\ast}$ and $X$ is an antirepresentation of $\widehat{\GG}$.
\end{proposition}
\begin{proof}
The proof is essentially the same as first part of the proof of \cite[Theorem 5.3]{MRW12}, but we repeat it for later use.

Denote $\tilde{X}=\ww_{12}^{\ast}\left((\id\otimes\,\theta)(\ww)\right)\in\mult{\Cnothat{G}}{\Cnot{G}}{\mathnormal{B}}$. We will show that $\ww_{23}\tilde{X}_{124}\ww_{23}^{\ast}=\tilde{X}_{134}$, then using \cite[Theorem 2.6]{MRW12} we conclude that $\tilde{X}\in\mult{\Cnothat{G}}{\mathbb{C}\mathds{1}}{\mathnormal{B}}$, so in fact there exists $X\in\mult{\Cnothat{G}}{\mathnormal{B}}$ with $\tilde{X}=X_{13}$. We compute

\begin{align*}
 \ww_{23}\tilde{X}_{124}\ww_{23}^{\ast}&=\ww_{23}\ww_{12}^{\ast}\ww_{23}^{\ast}\ww_{23}\bigg((\id\otimes\,\theta)\ww\bigg)_{124}\ww_{23}^{\ast}=\\
 &=\ww_{13}^{\ast}\ww_{12}^{\ast}(\id\otimes\Delta\otimes\id)\bigg((\id\otimes\,\theta)\ww\bigg)=\\
 &=\ww_{13}^{\ast}\ww_{12}^{\ast}\bigg((\id\otimes\id\otimes\,\theta)(\id\otimes\Delta)\ww\bigg)=\\
 &=\ww_{13}^{\ast}\ww_{12}^{\ast}\bigg((\id\otimes\id\otimes\,\theta)\ww_{12}\ww_{23}\bigg)=\\
 &=\ww_{13}^{\ast}\bigg((\id\otimes \id\otimes\,\theta)\ww_{23}\bigg)=\tilde{X}_{134}
\end{align*}
We now check that $\theta(a)=X(a\otimes\mathds{1})X^{\ast}$. This is equivalent to showing that $\theta(a)_{13}=\tilde{X}(a\otimes\mathds{1}\otimes\mathds{1})\tilde{X}^*$. We compute:
\begin{align*}
   \tilde{X}(a\otimes\mathds{1}\otimes\mathds{1})\tilde{X}^{\ast}&=\ww_{12}^{\ast}\bigg((\id\otimes\,\theta)\ww\bigg)(a\otimes\mathds{1}\otimes\mathds{1})\bigg((\id\otimes\,\theta)\ww^{\ast}\bigg)\ww_{12}=\\
   &=\ww_{12}^{\ast}\bigg((\id\otimes\,\theta)(\ww(a\otimes\mathds{1})\ww^{\ast})\bigg)\ww_{12}=\\
   &=\ww_{12}^{\ast}\bigg((\id\otimes\,\theta)\Delta(a)\bigg)\ww_{12}=\\
   &=\ww_{12}^{\ast}\bigg((\Delta\otimes \id)\theta(a)\bigg)\ww_{12}=\\
   &=\ww_{12}^{\ast}(\ww_{12}\theta(a)_{13}\ww_{12}^{\ast})\ww_{12}=\theta(a)_{13}
  \end{align*}
Showing that $X$ is antirepresentation amounts to showing that $(\widehat{\Delta}\otimes\id\otimes\id)\tilde{X}=\tilde{X}_{234}\tilde{X}_{134}$. Using \eqref{eq:dualcoproduct} and the definition of $\tilde{X}$ we get that:
\[\begin{split}
   (\widehat{\Delta}\otimes\id\otimes\id)\tilde{X}=&\biggl(\bigl((\widehat{\Delta}\otimes\id)\ww^*\bigr)\otimes\mathds{1}\biggr)(\widehat{\Delta}\otimes\,\theta)\ww)\\
   =&\ww_{13}^*\ww_{23}^*((\id\otimes\id\otimes\,\theta)\ww_{23})((\id\otimes\id\otimes\,\theta)\ww_{13})=\\
   =&\ww_{13}^*\tilde{X}_{234}(\widehat{\Delta}\otimes\,\theta)\ww)=\tilde{X}_{234}\tilde{X}_{134}
  \end{split}\]
where in the last equality we used the fact that $\tilde{X}_{234}=X_{24}$ commutes with $\ww_{13}^*$. To prove uniqueness, assume $Y\in\mult{\Cnot{G}}{\mathnormal{B}}$ is another such unitary. Because slices of $\ww$ are dense in $\Cnot{G}$, we have that \[X(a\otimes\mathds{1})X^*=Y(a\otimes\mathds{1})Y^*\]
for all $a\in\Cnot{G}$ is equivalent to saying that 
\[X_{23}\ww_{12}X_{23}^*=Y_{23}\ww_{12}Y_{23}^*\]
Rearranging terms, this is equivalent to
\[\ww_{12}^*(Y_{23}^*X_{23})\ww_{12}=Y_{23}^*X_{23}\]
which, in turn, is equivalent to 
\[(\widehat{\Delta}\otimes\id)(Y^*X)=Y_{13}^*X_{13}\]
We conclude from \cite[Corollary 2.9]{MRW12} that in fact $X(\mathds{1}\otimes u)=Y$ for some unitary $u\in\mult{\mathnormal{B}}$. Applying $(\widehat{\Delta}\otimes\id)$ to both sides of this equality we get
 \begin{align*}
X_{23}X_{13}(\mathds{1}\otimes\mathds{1}\otimes u)=(\widehat{\Delta}\otimes\id)\bigl(X(u\otimes\mathds{1})\bigr)&=(\widehat{\Delta}\otimes\id)Y\\
&=Y_{23}Y_{13}=X_{23}(\mathds{1}\otimes\mathds{1}\otimes u)X_{13}(\mathds{1}\otimes\mathds{1}\otimes u)  
 \end{align*}

 and hence $u=\mathds{1}$, which finishes the proof.
\end{proof}
Summarizing, there are three equivalent ways of studying an embedding $\mathbb{X}\hookrightarrow\GG$ of a locally compact quantum space into a locally compact quantum group (we recall that $\mathnormal{B}=\Cnot{X}$), these are as follows:
\begin{enumerate}
 \item the morphism $\beta\in \Mor(\Cunot{G},\mathnormal{B})$;
 \item the unitary $X\in\mult{\Cnothat{G}}{\mathnormal{B}}$, which is an antirepresentation of $\widehat{\GG}$ 
 \end{enumerate}
 
 and
 
 \begin{enumerate}\setcounter{enumi}{2}
 \item the morphism $\theta\in \Mor(\Cnot{G},\Cnot{G}\otimes\mathnormal{B})$ satisfying \[(\Delta\otimes\id)\comp\theta=(\id\otimes\,\theta)\comp\Delta,\] which corresponds to the partial action $\mathbb{X}\curvearrowright\GG$ by right shifts.
\end{enumerate}
Fixing a non-degenerate representation of $\mathnormal{B}\subseteq\B{H}$ and denoting by $\mathsf{B}=\mathnormal{B}''$ the WOT-closure of $\mathnormal{B}$ in the WOT-topology induced by this embedding, we can (similarly as in the case of ho\-mo\-mor\-phisms), study the above objects in the von Neumann algebraic context. Indeed, we have $X\in\Linfhat{G}\vnt\mathsf{B}\subseteq\Linfhat{G}\vnt\B{H}$ and $\theta\colon\Linf{G}\to\Linf{G}\vnt\mathsf{B}$ satisfying \[(\Delta\otimes\id)\comp\theta=(\id\otimes\,\theta)\comp\Delta,\] because $\theta$ is obtained by conjugating with a unitary and as such extends to the WOT-closure of $\Cnot{G}$. We will switch between these viewpoints freely later on. The passage from von Neumann level to $C^*$-level is pretty much the same as in \autoref{sec:LCQGrep}.

\section{The case of two subgroups}\label{sec:twosubgroups}
The goal of this section is to discuss the notion of \emph{compact quantum group generated by two closed quantum subgroups} in the sense of \cite{BCV}, as well as the notion of \emph{joint fullness of a family of (corestriction functors induced by) Hopf quotients} in the sense of \cite{ChirvaRFD}, in the context of Hopf image and to extend it to the non-compact case.

Let then $\GG$ be a locally compact quantum group and let $\HH_1, \HH_2$ be its two closed subgroups (for $i=1,2$, denote by $\pi_i\colon\Cunot{G}\to\Cunot{H}{i}$ the corresponding Hopf surjection, by $\gamma_i\colon\Linfhat{H}{i}\to\Linfhat{G}$ the corresponding inclusions and by $V^{\HH_i}\in\Linfhat{G}\vnt\Linf{H}{i}$ the corresponding bicharacters). Consider the two ideals: $C_0(\GG\setminus(\HH_1\cup\HH_2))\colon=\ker(\pi_1)\cap\ker(\pi_2)$ and  $C_0(\GG\setminus(\HH_1\cdot\HH_2))\colon=\ker((\pi_1\otimes\pi_2)\comp\Delta^u_{\GG})$ and the two quotients %
\begin{equation}\label{eq:unionofsubgroups}
q^{\cup}\colon\Cunot{G}\to C_0(\HH_1\cup\HH_2)=\bigslant{\Cunot{G}}{C_0(\mathbb{G}\setminus(\HH_1\cup\HH_2))} 
\end{equation}
 and 
 \begin{equation}\label{eq:productofsubgroups}
q^{\bullet}\colon\Cunot{G}\to C_0(\HH_1\cdot\HH_2)=\bigslant{\Cunot{G}}{C_0(\GG\setminus(\HH_1\cdot\HH_2))}      
\end{equation}
\begin{proposition}\label{prop:twosubgroups}
 The following von Neumann subalgebras of $\Linfhat{G}$ are equal:
 \begin{itemize}
  \item $\mathsf{M}_{1,2}$, the smallest von Neumann algebra containing both $\gamma_1(\Linfhat{H}{1})$ and $\gamma_2(\Linfhat{H}{2})$;
  \item $\mathsf{M}_{\bullet}=\{(\id\otimes \omega\comp q^{\bullet})\wW\colon\omega\in C_0(\HH_1\cdot\HH_2)^*\}''$;
  \item $\mathsf{M}_{\cup}=\{(\id\otimes \omega\comp q^{\cup})\wW\colon\omega\in C_0(\HH_1\cup\HH_2)^*\}''$.
  \item $\mathsf{M}_{V,1,2}=\{(\id\otimes\omega_1\otimes\omega_2)V^{\HH_1}_{12}V^{\HH_2}_{13}:\omega_i\in\Lone{H}{i}\}''$
 \end{itemize}

\end{proposition}
\begin{proof}
 $\mathsf{M}_{1,2}=\mathsf{M}_{\cup}$. Observe that $\ker(q^{\cup})=\ker(\pi_1\oplus\pi_2)$, hence arguing as in the proof of the first part of \autoref{lem:descriptionofM} we may replace $q^{\cup}$ with $\pi_1\oplus\pi_2$ in the definition of $\mathsf{M}_{\cup}$ (and the functionals are on a different $C^*$-algebra then). Recall that $\bigl(\Cunot{H}{1}\oplus\Cunot{H}{2})^*=\Cunot{H}{1}^*\oplus\Cunot{H}{2}^*$, hence every $\omega$ appearing in the definition of $\mathsf{M}_{\cup}$ can be written as $\omega=\omega_1\oplus\omega_2$ for $\omega_i\in\Cunot{H}{i}^*$. Hence \[\mathsf{M}_{\cup}=\biggl(\gamma_1(\Linfhat{H}{1})+\gamma_2(\Linfhat{H}{2})\biggr)''=\mathsf{M}_{1,2}\] as desired.
 
 $\mathsf{M}_{1,2}=\mathsf{M}_{\bullet}$. Recall that the linear span of the functionals of the form $\omega_1\otimes\omega_2$ on $\mathnormal{A}\otimes\mathnormal{B}$ is weak$^*$-dense in $(\mathnormal{A}\otimes\mathnormal{B})^*$ for any $C^*$-algebras $\mathnormal{A},\mathnormal{B}$. Further, as $(\id\otimes q^{\bullet})\wW=\bigl((\id\otimes\pi_1)\wW\bigr)_{12}\bigl((\id\otimes\pi_2)\wW\bigr)_{13}$, to compute $\mathsf{M}_{\bullet}$ it is enough to understand the von Neumann algebra generated by operators of the form $\bigl((\id\otimes\omega_1\comp\pi_1)\wW\bigr)\bigl((\id\otimes\omega_2\comp\pi_2)\wW\bigr)$ thanks to \autoref{lem:tensconv}. But it is then clear that
 \[\mathsf{M}_{\cup}=\biggl(\gamma_1(\Linfhat{H}{1})\cdot\gamma_2(\Linfhat{H}{2})\biggr)''=\mathsf{M}_{1,2}\]
 as desired. 
 
 $\mathsf{M}_{V,1,2}=\mathsf{M}_{1,2}$. Similarly as in the previous step, it is immediate to see that \[\mathsf{M}_{V,1,2}=\biggl(\gamma_1(\Linfhat{H}{1})\cdot\gamma_2(\Linfhat{H}{2})\biggr)''=\mathsf{M}_{1,2}\]
\end{proof}
\begin{df}\label{def:twosubgroupsgenerateJKS}
We call the Hopf image of either of the maps $q^{\bullet}$ and $q^{\cup}$ the closed quantum subgroup generated by $\mathbb{H}_1$ and $\mathbb{H}_2$ and denote it by $\overline{\langle\mathbb{H}_1,\mathbb{H}_2\rangle}$. More generally, if $\HH_i\subset\GG$ is a family of subgroups ($i\in I$), then we denote by $\overline{\langle\bigcup_{i\in I}\HH_i\rangle}$ the subgroup generated by this family: the subgroup corresponding to the Baaj-Vaes subalgebra $\bigl(\bigcup_{i\in I}\gamma_i(\Linfhat{H}{i})\bigr)''$.
\end{df}
\section{Separation of ho\-mo\-mor\-phisms}\label{sec:separation}
Let $\GG,\KK$ be locally compact quantum groups and let a ho\-mo\-mor\-phism of quantum groups $\GG\to\KK$ be described by a Hopf ${}^*$-ho\-mo\-mor\-phism $\varphi\in \Mor(\Cunot{K},\Cunot{G})$ and a bicharacter $V\in\mult{\Cnothat{K}}{\Cnot{G}}$. Let $\mathnormal{B}$ be a $C^{\ast}$-algebra, $\beta\in \Mor(\Cunot{G},\mathnormal{B})$, the corresponding unitary $X\in\mult{\Cnothat{G}}{\mathnormal{B}}$ is given by $X=(\id\otimes\beta)\wW^{\GG}$. Let us describe in detail the unitary corresponding to $\beta\comp\varphi$. It is given by $Y=(\id\otimes\beta\comp\varphi)\wW^{\KK}$.

\begin{lemma}\label{lem:composition}
The unitaries $X,Y,V$ obey the following equation:
\[ Y_{13}=V_{12}^{\ast}X_{23}V_{12}X_{23}^{\ast}\in\mult{\Cnothat{K}}{\mathnormal{K}(L^2(\GG))}{\mathnormal{B}}\]
\end{lemma}
\begin{proof}
 Let $\widehat{\varphi}\in\Mor(\Cunothat{G},\Cunothat{K})$ be the Hopf ${}^*$-ho\-mo\-mor\-phism linked to $\varphi$ by \eqref{eq:dualhomo}, that is, this Hopf ${}^*$-ho\-mo\-mor\-phism describes the ho\-mo\-mor\-phism of quantum groups $\widehat{\KK}\to\widehat{\GG}$ dual to the fixed ho\-mo\-mor\-phism $\GG\to\KK$. Application of $\Lambda_{\widehat{\KK}}\otimes\Lambda_{\GG}$ to both sides of \eqref{eq:dualhomo} yields
 \[V=(\id\otimes\Lambda_{\GG}\comp\varphi)\wW^{\KK}=(\Lambda_{\widehat{\KK}}\comp\widehat{\varphi}\otimes\id)\Ww^{\GG}\]
  Application of $\Lambda_{\widehat{\KK}}\comp\widehat{\varphi}\otimes\id\otimes\beta$ to both sides of \eqref{eq:universalpentagonal} gives:
 \[\begin{split}
   (\text{LHS})=&(\Lambda_{\widehat{\KK}}\comp\widehat{\varphi}\otimes\id\otimes\beta)\WW_{13}^{\GG}=(\id\otimes\id\otimes\beta)\bigg(\Lambda_{\widehat{\KK}}\comp\widehat{\varphi}\otimes \id\otimes\id)\WW_{13}^{\GG}\bigg)=\\
   =&(\id\otimes\id\otimes\beta)\bigg(\Lambda_{\widehat{\KK}}\otimes\id\otimes\varphi)\WW_{13}^{\GG}\bigg)=\bigg((\id\otimes\beta\comp\varphi)\wW^{\GG}\bigg)_{13}=Y_{13}
   \end{split}\]
and  
 \[\begin{split}
(\text{RHS})=&(\Lambda_{\widehat{\KK}}\comp\widehat{\varphi}\otimes\id\otimes\beta)\bigg((\Ww_{12}^{\GG})^{\ast}\wW_{23}^{\GG}\Ww_{12}^{\GG}(\wW_{23}^{\GG})^{\ast}\bigg)=\\
=&\bigg((\Lambda_{\widehat{\KK}}\comp\widehat{\varphi}\otimes\id)\Ww^{\GG}\bigg)_{12}^{\ast}\!\bigg((\id\otimes\beta)\wW^{\GG}\bigg)_{23}\!\bigg((\Lambda_{\widehat{\KK}}\comp\widehat{\varphi}\otimes\id)\Ww^{\GG}\bigg)_{12}\bigg((\id\otimes\beta)\wW^{\GG}\bigg)_{23}^{\ast}=\\
=&\bigg((\id\otimes\Lambda_{\GG}\comp\varphi)(\wW^{\KK})^{\ast}\bigg)_{12}X_{23}\bigg((\id\otimes\Lambda_{\GG}\comp\varphi)(\wW^{\KK})\bigg)_{12}X_{23}^{\ast}=V_{12}^{\ast}X_{23}V_{12}X_{23}^{\ast}
   \end{split}\]
\end{proof}
Recall the vague interpretation $\mathnormal{B}=C_0(\mathbb{X})$ and $\hat{\beta}\colon\mathbb{X}\hookrightarrow\GG$. With the notation as above, by the \emph{restriction of a ho\-mo\-mor\-phism $\GG\to\KK$ to the subset $\mathbb{X}\subset\GG$} we mean the map $\beta\comp\varphi\in\Mor(\Cunot{K},\mathnormal{B})$.

\begin{theorem}[{\autoref{thmintro:separation}}]\label{thm:homomorphismseparation}
 Assume $\beta$ is generating. Consider two quantum group ho\-mo\-mor\-phisms $\GG\to\KK$ described by $\varphi,\tilde{\varphi} \in \Mor(\Cunot{K},\Cunot{G})$. If their restrictions to $\mathbb{X}$ coincide, then the ho\-mo\-mor\-phisms coincide on the whole of $\GG$: $\varphi=\tilde{\varphi}$. 
\end{theorem}
\begin{proof}
 Let us denote by $V, \tilde{V}\in\mult{\Cnothat{K}}{\Cnot{G}}$ and by $\rho, \tilde{\rho}\in\Mor(\Cnot{K},\Cnot{G}\otimes\Cnot{K})$ the bicharacters and right quantum group ho\-mo\-mor\-phisms corresponding to $\varphi,\tilde{\varphi}$, respectively. 

Now, if $\beta\comp\varphi=\beta\comp\tilde{\varphi}$, then the corresponding unitaries coincide: $Y=\tilde{Y}\in\mult{\Cnothat{K}}{\mathnormal{B}}$. Using \autoref{lem:composition} one may rewrite this as
\[\tilde{V}^{\ast}_{12}X_{23}\tilde{V}_{12}X_{23}^{\ast}=V^{\ast}_{12}X_{23}V_{12}X_{23}^{\ast}\]
or equivalently
\[\tilde{V}^{\ast}_{12}X_{23}\tilde{V}_{12}=V^{\ast}_{12}X_{23}V_{12}.\]
By slicing with $\omega\in \mathnormal{B}^{\ast}$ on the third leg we see that the right quantum group ho\-mo\-mor\-phisms $\rho$ and $\tilde{\rho}$ agree on the $\mathsf{M}_1$ (they are normal ${}^{\ast}$-ho\-mo\-mor\-phisms). But applying $(\tau^{\widehat{\KK}}_t\otimes \tau_t^{\GG}\otimes\id)$ and remembering that $(\tau_t^{\widehat{\KK}}\otimes \tau^{\GG})(V)=V$ (and likewise for $\tilde{V}$) we see that $\rho$ and $\tilde{\rho}$ have the same values on $\widehat{\tau}_t(\mathsf{M}_1)$ and once again, by normality of $\rho, \tilde{\rho}$, they coincide on $(\bigcup_{t\in\mathbb{R}}\widehat{\tau}_t(\mathsf{M}_1))''=\mathsf{M}_{BV}$ (by \autoref{thm:BVclosure}). By assumption, $\mathsf{M}_{BV}=\Linfhat{G}$, hence $\rho=\tilde{\rho}$. This ends the proof, as $\rho$ and $\tilde{\rho}$ determine the ho\-mo\-mor\-phisms $\GG\to\KK$ uniquely (see \cite[Theorem 5.3]{MRW12}).
\end{proof}
\autoref{thm:homomorphismseparation} says that a ho\-mo\-mor\-phism between quantum groups is uniquely determined by its values on a quantum generating set. The rest of \autoref{sec:separation} is devoted to reversing the implication in case $\GG$ is discrete: if a subset has the property that ho\-mo\-mor\-phisms between quantum groups are uniquely determined by their values on this set, then the set in question is generating. In other words, we will manufacture two different ho\-mo\-mor\-phisms that agree on a given subset, provided that it is not generating. To this end, let $\beta\in\Mor(c_0(\GG), \mathnormal{B})$ be a morphism with Hopf image $(\pi,\HH,\tilde{\beta})$ such that $\HH\subsetneq\GG$. Denote by $\KK=\GG\ast_{\HH}\GG$. Using \cite[Theorem 3.4 \& Corollary 3.5]{Wang95}, one can describe $\KK$ as follows.

Consider first the  free product $\KK'=\GG\ast\GG$, it is given by the $C^{\ast}$-algebra $C^u(\widehat{\KK'})=C^u(\widehat{\GG})\ast C^u(\widehat{\GG})$ (amalgamated over $\mathbb{C}\mathds{1}$). Denote by $i_1, i_2$ the maps $C^u(\widehat{\GG})\to C^u(\widehat{\GG})\ast C^u(\widehat{\GG})=C^u(\widehat{\KK'})$ putting the copy of $C^u(\widehat{\GG})$ in the first and second spot, respectively, these maps are Hopf morphisms. Denote by $\widehat{\pi}\colon C^u(\widehat{\HH})\to C^u(\widehat{\GG})$ the ho\-mo\-mor\-phism dual to $\pi$.
\begin{lemma}\label{lem:notsurjective}
 $\widehat{\pi}$ is not surjective.
\end{lemma}
\begin{proof}
 Assume it is. Using \eqref{eq:reduceddualhomo1} we then have that $\gamma\colon\Linfhat{H}\to\Linfhat{G}$ is surjective, which contradicts our assumption $\HH\neq\GG$.
\end{proof}
Denote by $I\subseteq C^u(\widehat{\KK'})$ the closed ideal generated by $\{i_1\!\circ\!\widehat{\pi}(x)-i_2\!\circ\!\widehat{\pi}(x): x\in C^u(\widehat{\HH})\}$. 
\begin{lemma}\label{lem:smallideal}
There exists $y\in C^u(\GG)$ such that $i_1(y)-i_2(y)\notin I$.
\end{lemma}
\begin{proof}
 Consider the smallest $C^*$-subalgebra of $C^u(\widehat{\KK'})$ generated by $i_1(C^u(\widehat{\GG}))$ and $i_2(\widehat{\pi}(C^u(\widehat{\HH}))$ inside $C^u(\widehat{\KK'})$: let it be called $\mathnormal{A}$. It follows from the concrete description of the free product (see, e.g. \cite[\S4.7]{BrownOzawa}) that $\mathnormal{A}\neq C^u(\widehat{\KK'})$ (this of course needs \autoref{lem:notsurjective}). Let then $\omega\in C^u(\widehat{\KK'})^*$ be a non-zero functional such that $\omega\restriction_{\mathnormal{A}}=0$. Then in particular $\omega\restriction_I=0$. Let $y\in C^u(\widehat{\GG})\setminus\widehat{\pi}(C^u(\widehat{\HH}))$. Then \[\omega(i_1(y)-i_2(y))=-\omega(i_2(y))\]
 It is now clear that for a given $y$ as above one could have chosen $\omega$ such that $\omega(i_2(y))\neq0$ and $\omega\restriction_{\mathnormal{A}}=0$. Indeed, first pick non-zero $\tilde{\omega}\in C^u(\widehat{\GG})$ such that $\tilde{\omega}\restriction_{\widehat{\pi}(C^u(\widehat{\HH}))}=0$ and $\tilde{\omega}(y)\neq0$. Then $\omega=\tilde{\omega}\comp i_2\in C^u(\widehat{\KK'})^*$ is the required functional.
\end{proof}
As the final step of constructing $\KK=\GG\ast_{\HH}\GG$, one considers $q\colon C(\widehat{\KK'})\to C(\widehat{\KK})=\bigslant{C(\widehat{\KK'})}{I}$. Consider the morphisms $\varphi_j=\widehat{q\comp i_j}\colon c_0(\KK)\to c_0(\GG)$ for $j=1,2$. Then one has
\begin{proposition}\label{prop:homomorphismsconverse}
 The Hopf ${}^*$-ho\-mo\-mor\-phisms $\varphi_j$ satisfy $\beta\comp\varphi_1=\beta\comp\varphi_2$ and are distinct.
\end{proposition}
\begin{proof}
 Observe that it is enough to show that $\varphi_j$ coincide on $\mathbb{H}$, i.e.~
 \begin{equation}\label{eq:homoagree1}
\pi\comp\varphi_1=\pi\comp\varphi_2  
 \end{equation}because $\beta=\tilde{\beta}\comp\pi$. The equality \eqref{eq:homoagree1} is equivalent to the equality 
\begin{equation}\label{eq:homoagree2}
 \widehat{\pi\comp\varphi_1}=\widehat{\pi\comp\varphi_2}
 \end{equation}
But composition of morphisms satisfies $\widehat{\varphi\comp\phi}=\widehat{\phi}\comp\widehat{\varphi}$, hence \eqref{eq:homoagree2} is equivalent to
\begin{equation}\label{eq:homoagree3}
 q\comp i_1\comp\widehat{\pi}=\widehat{\varphi_1}\comp\widehat{\pi}=\widehat{\varphi_2}\comp\widehat{\pi}=q\comp i_2\comp\widehat{\pi}
 \end{equation}
 which holds in the quotient $C(\widehat{\KK})=\bigslant{C(\widehat{\KK'})}{I}$, as desired.
 
 Now using $y\in C^u(\widehat{\GG})$ from \autoref{lem:smallideal} we can see that $\varphi_1\neq\varphi_2$. Indeed, 
 \[\widehat{\varphi}_1(y)-\widehat{\varphi}_2(y)=q((i_1-i_2)(y))\neq0\]
 from the definition of $y$ and hence $\varphi_1\neq\varphi_2$.
\end{proof}
Let us summarize these ideas in the following
\begin{theorem}\label{thm:separationconverse}
Let $\GG$ be a discrete quantum group and let $\beta\in\Mor(c_0(\GG),\mathnormal{B})$ be a morphism. Then $\beta$ is generating if and only if for any quantum group $\KK$ and any pair of ho\-mo\-mor\-phisms $\GG\to\KK$ described by $\varphi,\tilde{\varphi}\in\Mor(c_0(\KK),c_0(\GG))$ one has $\beta\comp\varphi=\beta\comp\tilde{\varphi} \iff \varphi=\tilde{\varphi}$ .
\end{theorem}

\section{\texorpdfstring{$\beta$}{beta}-restriction and generating morphisms}\label{sec:betarestriction}
Let $\mathnormal{A}$ be a $C^{\ast}$-algebra and let $U\in\mult{\mathnormal{A}}{\Cnot{G}}$ be a representation of $\GG$ on $\mathnormal{A}$. Reasoning similarly as in the proof of \autoref{prop:MRWlike} (with certain $\ww$'s replaced by $U$), one can prove the following proposition:
\begin{proposition}\label{prop:betarestriction}
 Let $\theta\in \Mor(\Cnot{G},\Cnot{G}\otimes\mathnormal{B})$ be a morphism satisfying \[(\Delta_{\mathbb{G}}\otimes \id)\comp\theta=(\id\otimes\,\theta)\comp\Delta_{\mathbb{G}}.\] There exists a unique unitary element $Y\in\mult{\mathnormal{A}}{\mathnormal{B}}$ such that $(\id\otimes\,\theta)U=U_{12}Y_{13}$. 
\end{proposition}
\begin{df}\label{def:betarestriction}
Let $\beta\in \Mor(\Cunot{G},\mathnormal{B})$ be a morphism, $\theta\in \Mor(\Cnot{G},\Cnot{G}\otimes\mathnormal{B})$ be the morphism assigned to $\beta$ via \eqref{eq:partialcoaction1}, let $U\in\mult{\mathnormal{A}}{\Cnot{G}}$ be a representation of $\GG$ in $\mathnormal{A}$. Then $Y\in\mult{\mathnormal{A}}{\mathnormal{B}}$ obtained by \autoref{prop:betarestriction} will be denoted $U^{\beta}$ and called the \emph{$\beta$-restriction} of $U$.
\end{df}
\begin{remark}\label{rmk:Wbeta} Observe that \eqref{eq:partialcoaction2} may be interpreted as $\ww^{\beta}=X=(\id\otimes\beta)\wW$. 
\end{remark}

\begin{theorem}[{\autoref{thmintro:betarestriction}}]\label{thm:characterisation}
With the notation as above
\begin{enumerate}[label={\normalfont(\roman*)}]
  \item \label{it:char1} The morphism $\beta$ is generating.
  \item \label{it:char2} $\{(\id\otimes\beta_*\omega)\wW \mid \omega\in \mathnormal{B}^*\}''=\Linfhat{G}$.
  \item \label{it:char3} $\{x\in L^{\infty}(\mathbb{G}) \mid \theta(x)=x\otimes\mathds{1}\}=\mathbb{C}\mathds{1}$.
  \item \label{it:char4} $\beta$-restriction is an injective assignment: $\beta$-restrictions of any two distinct representations of $\GG$ are distinct.
   \end{enumerate}
 We have \ref{it:char1}$\isimplied$\ref{it:char2}$\iff$\ref{it:char3}$\iff$\ref{it:char4}. Moreover, \ref{it:char1}$\implies$\ref{it:char2} provided that $\GG$ is compact or discrete or $\tau^u_t(\ker(\beta))\subseteq\ker(\beta)$ for all $t\in\mathbb{R}$ (in particular, if $\GG$ is of Kac type).
\end{theorem}
 \begin{proof}
 \ref{it:char2}$\implies$\ref{it:char1}. This is obvious in view of \autoref{thm:BVclosure}.
 
 \ref{it:char2}$\implies$\ref{it:char3}. This follows by applying co-duality to both sides of the equation in \ref{it:char2}, see \cite[Section 3]{KS14a}. An elementary proof is given in \cite[Theorem 3.15]{PJphd}.
 
 \ref{it:char3}$\implies$\ref{it:char4}. Let $U,V\in\mult{\K{H}}{\Cnot{G}}$ be two representations of $\GG$ in the same Hilbert space $\mathcal{H}$. Assume that $U^{\beta}=V^{\beta}\in\mult{\K{H}}{\mathnormal{B}}$. We have:
 \[(\id\otimes\,\theta)(UV^*)=U_{12}U_{13}^{\beta}(V_{13}^{\beta})^*V_{12}^*=U_{12}V_{12}^*.\]
 Thus condition \ref{it:char2} ensures us that there exists a unitary element $u\in\B{H}$ such that $U=(u\otimes\mathds{1})V$. Applying $(\id\otimes\Delta)$ to both sides of this equality we get that $u=\mathds{1}$ as in the last step of the proof of \autoref{prop:MRWlike}.
 
 \ref{it:char4}$\implies$\ref{it:char2}. Assume that \ref{it:char2} does not hold, i.e.~we have that $\mathsf{M}_1\subsetneq\Linfhat{G}$. Then we have $\Linfhat{G}'\subsetneq\mathsf{M}_1'\subset\BLtwo$, so pick a unitary $u\in\mathsf{M}_1'\setminus\Linfhat{G}'$ (it exists, as von Neumann algebras are spanned by their unitary elements). Consider $U=(u\otimes\mathds{1})\ww(u^{\ast}\otimes\mathds{1})$ (it is obvious that $U\in \Rep(\mathbb{G})$). From the definition of $u$ it is clear that $U\neq\ww$. But on the other hand we have that 
  \[\begin{split}
   U^{\beta}_{13}=&U_{12}^{\ast}((\id\otimes\,\theta)U)=\\
   =&U_{12}^{\ast}\bigg((\id\otimes\,\theta)\bigg((u\otimes\mathds{1})\ww(u^{\ast}\otimes\mathds{1})\bigg)\bigg)=\\
   =&(u\otimes\mathds{1}\otimes\mathds{1})\ww_{12}^{\ast}(u^{\ast}\otimes\mathds{1}\otimes\mathds{1})(u\otimes\mathds{1}\otimes\mathds{1})\bigg((\id\otimes\,\theta)\ww\bigg)(u^{\ast}\otimes\mathds{1}\otimes\mathds{1})=\\
   =&(u\otimes\mathds{1}\otimes\mathds{1})\ww^{\beta}_{13}(u^{\ast}\otimes\mathds{1}\otimes\mathds{1})=\\
   =&(u\otimes\mathds{1}\otimes\mathds{1})X_{13}(u^{\ast}\otimes\mathds{1}\otimes\mathds{1})=X=\ww^{\beta}_{13}
  \end{split}\]
  Where the equalities in the last line follow from the fact that $\ww^{\beta}=X$ (cf. \autoref{rmk:Wbeta}) and the fact that $X\in\mathsf{M}_1\vnt\B{H}$ (for a fixed non-degenerate representation $\pi\in\Mor(\mathnormal{B},\K{H})$) and hence the first leg of $X$ commutes with $u$.
  
  \ref{it:char1}$\implies$\ref{it:char2} under additional assumptions. This was discussed in \autoref{prop:Misinvariant} and \autoref{thm:BVclosure} and boils down to $\mathsf{M}_1$ being automatically $\tau$-invariant.
 \end{proof}
 \begin{remark}
  In fact \autoref{thm:homomorphismseparation} can be deduced from \autoref{thm:characterisation} in the case of compact, discrete and Kac type quantum groups: it relies on the implication \ref{it:char4}$\implies$\ref{it:char3} for representations coming from bicharacters describing ho\-mo\-mor\-phisms. Its converse, \autoref{thm:separationconverse}, valid for discrete quantum groups, for which \ref{it:char4}$\implies$\ref{it:char3} holds automatically, is stronger than just this implication: one can detect injectivity of the $\beta$-restriction map not only in the class of all representations, but also in the class of bicharacters coming from ho\-mo\-mor\-phisms.
 \end{remark}
\section{Promotion of intertwiners}\label{sec:intertwiners}
Let $U\in\mult{\K{H}{U}}{\Cnot{G}}\subseteq\B{H}{U}\vnt\Linf{G}$ be a representation of $\GG$. Let us call $\mathnormal{B}$ the $C^*$-subalgebra of $\B{H}{U}$ generated by slices of $U$, then $U\in\mult{\mathnormal{B}}{\Cnot{G}}\subseteq \mathsf{B}\vnt\Linf{G}$, where the bicommutant $\mathsf{B}=\mathnormal{B}''$ is taken inside $\B{H}{U}$. Let $\varphi\in\Mor(\Cunothat{G},\mathnormal{B})$ be the unique morphism such that $U=(\varphi\otimes\id)\Ww$ (given by \autoref{thm:kustermans}).

We would like to interpret $X=\flip(U)^*\in\mult{\Cnot{G}}{\mathnormal{B}}$ (which is now an antirepresentation) as a quantum subset $\mathbb{X}\subset\widehat{\mathbb{G}}$. Let then $\theta\colon\Linfhat{G}\to\Linfhat{G}\vnt\mathsf{B}$ be the map given by \eqref{eq:partialcoaction1}. 

Also, for any representation $\tilde{U}\in\mult{\mathnormal{A}}{\Cnothat{G}}$ of $\widehat{\GG}$ there exists a unique $\tilde{U}^{\varphi}\in\mult{\mathnormal{A}}{\mathnormal{B}}$ -- the restriction of the family of unitaries $\tilde{U}$ to the quantum subset $\mathbb{X}\subset\widehat{\mathbb{G}}$ (given by \autoref{prop:betarestriction}). The precise formula is \begin{equation}\label{eq:restriction0}
\tilde{U}^{\varphi}_{13}=\tilde{U}_{12}^*(\id\otimes\,\theta)(\tilde{U}).\end{equation}

Let us use the following notation: whenever $\pi\in\Mor(\Cunot{G},\Cunot{H})$ is a Hopf ${}^*$-ho\-mo\-mor\-phism, we denote by $\pi^r=\Lambda_{\mathbb{H}}\comp\pi\in\Mor(\Cunot{G},\Cnot{H})$ the reduction of this morphism. We begin with computing the restriction of $U\in\mult{\mathnormal{B}}{\Cnot{G}}$ to the subset $\HH\subset\GG$, where $\HH$ is a closed quantum subgroup (via $\pi\colon\Cunot{G}\to\Cunot{H}$, normal embedding $\gamma\colon\Linfhat{H}\to\Linfhat{G}$ and with the aid of the bicharacter $V\in\mult{\Cnothat{H}}{\Cnot{G}}$). 

\begin{align}\label{eq:Urestrict}
\begin{split}
U^{\pi^r}_{13}&=U_{12}^*V_{23}U_{12}^*V_{23}^*\\
&=(\varphi\otimes\id\otimes\Lambda\comp\pi)(\Ww^*_{12}\wW_{23}\Ww_{12}\wW_{23}^*)\\
 &=\bigl((\varphi\otimes\Lambda\comp\pi)\WW\bigr)_{13}=\bigl((\varphi\comp\hat{\pi}\otimes\id)\Ww^{\HH}\bigr)_{13}
 \end{split}
\end{align}
In the first equality we used \eqref{eq:restriction0} and \eqref{eq:universalpentagonal} in the third equality. In fact the above computation works also for Woronowicz-closed quantum subgroups. For later use, let us compute $(\theta\otimes\id)V$. 
\begin{align}\label{eq:thetaonV}
 \begin{split}
(\theta\otimes\id)V=X_{12}V_{13}X_{12}^*&=(\flip\otimes\id)(U_{12}^*V_{23}U_{12})
\\ &=(\flip\otimes\id)\bigl((\varphi\otimes\id\otimes\Lambda\comp\pi)\Ww_{12}^*\wW_{23}\Ww_{12}\bigr)
\\ &=(\flip\otimes\id)\bigl((\varphi\otimes\id\otimes\Lambda\comp\pi)(\WW_{13}\wW_{23}\bigr)
\\ &=\bigl((\varphi\otimes\Lambda\comp\pi)\WW\bigr)_{23}V_{13}=U^{\pi^r}_{23}V_{13}
 \end{split}
\end{align}
In particular if $\mathbb{H}=\mathbb{G}$ and $V=\ww$, then \eqref{eq:thetaonV} simply says that $(\theta\otimes\id)\ww=U_{23}\ww_{13}$ (cf. \eqref{eq:partialcoaction2}). Let us denote by $\tilde{\theta}$ the map given by \eqref{eq:partialcoaction1}, associated to the $U^{\pi^r}$, a representation of $\HH$.

\begin{lemma}\label{lem:restrictcoincide}
The maps $(\gamma\otimes\id)\comp\tilde{\theta}$ and $\theta\comp\gamma$ coincide on $\Linfhat{H}$.
\end{lemma}
\begin{proof}
Using \eqref{eq:thetaonV} we have
\begin{equation}\label{eq:thetaonV2}
\bigl((\theta\comp\gamma)\otimes\id\bigr)(\ww^{\HH})=(\theta\otimes\id)V=U^{\pi^r}_{23}V_{13}
 \end{equation}
On the other hand using \eqref{eq:Urestrict} we obtain:
\begin{equation}\label{eq:tildethetaonW}
\begin{split}
(\flip\otimes\id)\comp(\tilde{\theta}\otimes\id)(\ww^{\HH})&=(U^{\pi^r})_{12}^*\ww_{23}^{\HH}U^{\pi^r}_{12}\\
&=(\varphi\comp\hat{\pi}\otimes\id\otimes\id)\bigl((\Ww^{\HH}_{12})^*\ww^{\HH}_{23}\Ww^{\HH}_{12}\bigr)\\
&=(\varphi\comp\hat{\pi}\otimes\id\otimes\id)\bigl((\Ww^{\HH}_{13})\ww^{\HH}_{23}\bigr)=U^{\pi^r}_{13}\ww^{\HH}_{23},
\end{split}
\end{equation}
where in the third equality we used \eqref{eq:universalpentagonal} (after reducing the third leg) . Application of $(\flip\otimes\id)$ to both sides of \eqref{eq:tildethetaonW} yields:
\begin{equation}\label{eq:tildethetaonW2}
(\tilde{\theta}\otimes\id)(\ww^{\HH})=U^{\pi^r}_{23}\ww^{\HH}_{13},
\end{equation}
Applying $(\gamma\otimes\id)$ to both sides of \eqref{eq:tildethetaonW2} and comparing it to \eqref{eq:thetaonV2} we obtain:
\begin{equation}
\bigl((\theta\comp\gamma)\otimes\id\bigr)(\ww^{\HH})=(\gamma\otimes\id)\comp(\tilde{\theta}\otimes\id)(\ww^{\HH})
\end{equation}
and the conclusion follows from WOT-density of slices of $\ww^{\HH}$ in $\Linfhat{H}$.
\end{proof}

\begin{theorem}\label{thm:promtheta}
 Let $\HH_1,\HH_2\subset\GG$ be closed quantum subgroups of a locally compact quantum group $\GG$ identified via $\pi_i\in\Mor(\Cunot{G},\Cunot{H}{i})$. Then the following conditions are equivalent:
\begin{enumerate}[label={\normalfont(\roman*)}]
 \item\label{it:promtheta1} $\GG=\overline{\langle \HH_1,\HH_2\rangle}$ (in the sense of \autoref{sec:twosubgroups});
 \item\label{it:promtheta2} for all representations $U\in\B{H}{U}\vnt\Linf{G}$ of $\GG$ we have that \begin{equation}\label{eq:promtheta2}\{(\id\otimes\omega)(U):\omega\in \Lone{G}\}''=\{(\id\otimes\omega_1\otimes\omega_2)(U^{\pi_1^r}_{12}U^{\pi^r_2}_{23}):\omega_i\in\Lone{H}{i}\}'';\end{equation}
 \item\label{it:promtheta3} for the right regular representation $\ww\in\BLtwo\vnt\Linf{G}$ we have that \begin{equation}\label{eq:promtheta3}\{(\id\otimes\omega)(\ww):\omega\in \Lone{G}\}''=\{(\id\otimes\omega_1\otimes\omega_2)(\ww^{\pi^r_1}_{12}\ww^{\pi^r_2}_{23}):\omega_i\in \Lone{H}{i}\}''.\end{equation}
\end{enumerate}\end{theorem}
\begin{proof}

\ref{it:promtheta1}$\implies$\ref{it:promtheta2}. Assume that $\GG=\overline{\langle\HH_1,\HH_2\rangle}$, where the embedding $\HH_i\subset\GG$ is described by means of a Hopf ${}^*$-homomorphism $\pi_i\in\Mor(\Cunot{G},\Cunot{H}{i})$ and a bicharacter $V^{\HH_i}\in\Linfhat{G}\vnt\Linf{H}{i}$. From \autoref{prop:twosubgroups} this is to say that
 \begin{equation}\label{eq:generationdef}\Linfhat{G}=\{(\id\otimes\omega)\ww:\omega\in \Lone{G}\}''=\{(\id\otimes\omega_1\otimes\omega_2)V^{\HH_1}_{12}V^{\HH_2}_{13}:\omega_i\in \Lone{H}{i}\}''                                                                                                                                                                                                                                                                                                                      \end{equation}

 Let us fix a representation $U\in\mathsf{B}\vnt\Linf{G}$ of $\GG$ interpret it as a quantum subset $\mathbb{X}\subset\widehat{\mathbb{G}}$ as in the introduction. Let then $\theta\colon\Linfhat{G}\to\Linfhat{G}\vnt\mathsf{B}$ be the corresponding morphism. Let us apply the map $\theta$ to middle and right hand side of \eqref{eq:generationdef}. The right hand side is then 
 \begin{align}\label{eq:RHS1}\begin{split}
\{(\theta\otimes\omega_1\otimes\omega_2)(V^{\mathbb{H}_1}_{12}V^{\mathbb{H}_2}_{13}):\omega_i\in\Lone{H}{i}\}''&=\{(\id\otimes\id\otimes\omega_1\otimes\omega_2)(U^{\pi^r_1}_{23}V^{\mathbb{H}_1}_{13}U^{\pi^r_2}_{24}V^{\mathbb{H}_2}_{14}):\omega_i\in\Lone{H}{i}\}''\\
&=\{(\id\otimes\id\otimes\omega_1\otimes\omega_2)(U^{\pi^r_1}_{23}U^{\pi^r_2}_{24}V^{\mathbb{H}_1}_{13}V^{\mathbb{H}_2}_{14}):\omega_i\in\Lone{H}{i}\}''
               \end{split}\end{align}
whereas the left (middle) hand side is 
 \begin{align}\label{eq:LHS1}
\{(\theta\otimes\omega)(\ww):\omega\in\Lone{G}\}''=\{(\id\otimes\id\otimes\omega)(U_{23}\ww_{13}):\omega\in\Lone{G}\}''
\end{align}
Applying $(\eta\otimes\id)$ for $\eta\in\Lonehat{G}$ to all elements appearing in \eqref{eq:LHS1} and \eqref{eq:RHS1}, by normality, yields:
\[\begin{split}\{(\eta\otimes\id\otimes\omega_1\otimes\omega_2)(U^{\pi^r_1}_{23}U^{\pi^r_2}_{24}V^{\mathbb{H}_1}_{13}V^{\mathbb{H}_2}_{14}):\omega_i\in\Lone{H}{i}\}''\\
   =\{(\eta\otimes\id\otimes\omega)(U_{23}\ww_{13}):\omega\in\Lone{G}\}''
  \end{split}
\]
Now letting $\eta$ run through the whole set $\Lonehat{G}$, we obtain
\begin{equation}\label{eq:step}
\begin{split}\{(\eta\otimes\id\otimes\omega_1\otimes\omega_2)(U^{\pi^r_1}_{23}U^{\pi^r_2}_{24}V^{\mathbb{H}_1}_{13}V^{\mathbb{H}_2}_{14}):\omega_i\in\Lone{H}{i},\eta\in\Lonehat{G}\}''\\
   =\{(\eta\otimes\id\otimes\omega)(U_{23}\ww_{13}):\omega\in\Lone{G},\eta\in\Lonehat{G}\}''
  \end{split} 
\end{equation}
Remembering that $(\eta\otimes\id)\ww$ generate $\Cnot{G}$ and that $\Cnot{G}\subseteq\BLtwo$ is nondegenerate, observe that the natural action of $\Cnot{G}$ on $\BLtwo_*$ is non-degenerate (cf. \cite[eq. (1.2)]{DKSS}), hence the right-hand side of \eqref{eq:step} reads as:
\begin{equation}\label{eq:RHS2}\begin{split}
\{(\eta\otimes\id\otimes\omega)(U_{23}\ww_{13}):\omega\in\Lone{G},\eta\in\Lonehat{G}\}''\\
 =\{(\eta\otimes\id\otimes\omega)(U_{23}):\omega\in\Lone{G},\eta\in\Lonehat{G}\}''\\
 =\{(\id\otimes\omega)(U):\omega\in\Lone{G}\}''                                
                               \end{split}\end{equation}
and similarly the left-hand side of \eqref{eq:step} turns into
 \begin{equation}\label{eq:LHS2}
  \begin{split}\{(\eta\otimes\id\otimes\omega_1\otimes\omega_2)(U^{\pi^r_1}_{23}U^{\pi^r_2}_{24}V^{\mathbb{H}_1}_{13}V^{\mathbb{H}_2}_{14}):\omega_i\in\Lone{H}{i},\eta\in\Lonehat{G}\}''\\
  =\{(\eta\otimes\id\otimes\omega_1\otimes\omega_2)(U^{\pi^r_1}_{23}U^{\pi^r_2}_{24}):\omega_i\in\Lone{H}{i},\eta\in\Lonehat{G}\}''\\
  =\{(\id\otimes\omega_1\otimes\omega_2)(U^{\pi^r_1}_{12}U^{\pi^r_2}_{23}):\omega_i\in\Lone{H}{i}\}''
     \end{split}
 \end{equation}
since $V^{\HH_i}$ generates $\Cnot{H}{i}$ (see \cite[Definition 3.2, Theorem 3.5 \& Theorem 3.4]{DKSS}). Combining \eqref{eq:step} with \eqref{eq:LHS2} and \eqref{eq:RHS2}, we obtain \eqref{eq:promtheta2}.
\ref{it:promtheta2}$\implies$\ref{it:promtheta3} is obvious, as one specializes $U=\ww$, whereas \ref{it:promtheta3}$\implies$\ref{it:promtheta1} was already shown in \autoref{prop:twosubgroups} (and used as the starting point of the implication \ref{it:promtheta1}$\implies$\ref{it:promtheta2}). 
\end{proof}

Recall that for two representations $U\in\B{H}{U}\vnt\Linf{G}$ and $\tilde{U}\in\B{H}{\tilde{U}}\vnt\Linf{G}$ of a locally compact quantum group $\GG$ we denote by $\Hom_{\GG}(U,\tilde{U})=\{t\in\mathsf{B}(\mathcal{H}_U,\mathcal{H}_{\tilde{U}}):(t\otimes\mathds{1})U=\tilde{U}(t\otimes\mathds{1})\}$ the set of intertwiners between $U$ and $\tilde{U}$. Then one obviously has $\Hom_{\mathbb{G}}(U,\tilde{U})\subseteq\Hom_{\HH}(U^{\pi^r},\tilde{U}^{\pi^r})$ for every closed quantum subgroup $\HH\subset\GG$.
\begin{lemma}\label{lem:selfintertwiners}
With the notation as above, let $t\in\B{H}{U}$. Then $t\in\Hom_{\GG}(U,U)$ if and only if $(\mathds{1}\otimes t)\theta(a)=\theta(a)(\mathds{1}\otimes t)$ for all $a\in\Linfhat{G}$.
\end{lemma}
\begin{proof}
 $(\implies)$ This follows from the following simple computation, with $a\in\Linfhat{G}$ and $t\in\Mor_{\GG}(U,U)$:
 \[\flip\biggl((\mathds{1}\otimes t)\theta(a)\biggr)=(t\otimes\mathds{1})U(\mathds{1}\otimes a)U^*=U(\mathds{1}\otimes a)(t\otimes\mathds{1})U^*=U(\mathds{1}\otimes a)U^*(t\otimes\mathds{1})=\flip\biggl(\theta(a)(\mathds{1}\otimes t)\biggr)\]
 
 $(\isimplied)$ It is clear (and follows from WOT-density of slices of $\ww^{\GG}=\ww$ in $\Linfhat{G}$) that the assumption is equivalent to:
 \begin{equation}\label{eq:selfintertwiners1} (\mathds{1}\otimes t\otimes\mathds{1})(\theta\otimes\id)(\ww)=(\theta\otimes\id)(\ww)(\mathds{1}\otimes t\otimes\mathds{1})\end{equation}
and using \eqref{eq:partialcoaction1} one develops \eqref{eq:selfintertwiners1} to obtain: 
\begin{equation}
 t_1U_{12}^*\ww_{23}U_{12}=U_{12}^*\ww_{23}U_{12}t_1
\end{equation}
which can be rewritten as
\begin{equation}\label{eq:selfintertwiners2}
U_{12}t_1U_{12}^*=\ww_{23}U_{12}t_1U_{12}^*\ww_{23}^* 
\end{equation}
Observe that the right hand side of \eqref{eq:selfintertwiners2} is nothing but $(\id\otimes\Delta)(U_{12}t_1U_{12}^*)$, so using ergodicity of coproduct (see, e.g. \cite[Theorem 2.6]{MRW12}) we have that
\[U_{12}t_1U_{12}^*\in\B{H}{U}\vnt\mathbb{C}\mathds{1}\vnt\mathbb{C}\mathds{1}\]
Let then $m\in\mathsf{B}$ be such that:
\begin{equation}\label{eq:selfintertwiners3}
U(t\otimes\mathds{1})U^*=m\otimes\mathds{1} 
\end{equation}
and write this as
\begin{equation}\label{eq:selfintertwiners4}U(t\otimes\mathds{1})=(m\otimes\mathds{1})U\end{equation}
Applying $(\id\otimes\Delta)$ to both sides of \eqref{eq:selfintertwiners4} and rearranging the terms one arrives at
\begin{equation}\label{eq:selfintertwiners5}
\B{H}{U}\vnt\mathbb{C}\mathds{1}\vnt\Linf{G}\ni U_{13}t_1U_{13}^*=U_{12}m_1U_{12}^*\in\B{H}{U}\vnt\Linf{G}\vnt\mathbb{C}\mathds{1}
\end{equation}
Hence it follows (see, e.g. \cite[12.4.36]{KadisonRingrose4}) that
\begin{equation}
m_1=U_{13}t_1U_{13}^*=U_{12}m_1U_{12}^*
\end{equation}
and consequently $m\otimes\mathds{1}$ commutes with $U$ and hence $t=m$, which finishes the proof in view of \eqref{eq:selfintertwiners3}.
\end{proof}
\begin{theorem}[{\autoref{thmintro:promofintertwin}}]\label{thm:promintertwin}
 Let $(\HH_i)_{i\in I}\subset\GG$ be a family of closed subgroups. Then $\GG$ is generated by $(\HH_i)_{i\in I}$ if and only if for all representations $U\in\B{H}{U}\vnt\Linf{G}$ and $\tilde{U}\in\B{H}{\tilde{U}}\vnt\Linf{G}$ we have that
 \begin{equation}\label{eq:promintertwin}\Hom_{\GG}(U,\tilde{U})=\bigcap_{i\in I}\Hom_{\HH_i}(U^{\pi^r_i},\tilde{U}^{\pi^r_i})\end{equation}
\end{theorem}

\begin{proof}
It is clear (and noted before \autoref{lem:selfintertwiners}) that \[\Hom_{\GG}(U,\tilde{U})\subseteq\bigcap_{i\in I}\Hom_{\HH_i}(U^{\pi^r_i},\tilde{U}^{\pi^r_i})\]
hence the genuine statement is to obtain the converse containment, under assumption that $\GG$ is generated by $(\HH_i)_{i\in I}$.

Assume first that $U=\tilde{U}$ and $t\in\bigcap_{i\in I}\Hom_{\HH_i}(U^{\pi^r_i},U^{\pi^r_i})$. Then \[(t\otimes\mathds{1})\tilde{\theta}^i(a)=\tilde{\theta}^i(a)(t\otimes\mathds{1})\] for all $i\in I$ and $a\in\Linfhat{H}{i}$, where $\tilde{\theta}^i$ is as \autoref{lem:restrictcoincide} for $\HH_i\subset\GG$. But \autoref{lem:restrictcoincide}, together with the assumption $\Linfhat{G}=\bigl(\bigcup_{i\in I}\gamma_i(\Linfhat{H}{i})\bigr)''$ and normality of all the maps involved gives us that \[(t\otimes\mathds{1})\theta(a)=\theta(a)(t\otimes\mathds{1})\] for all and $a\in\Linfhat{G}$, hence we conclude by \autoref{lem:selfintertwiners} that $t\in\Hom_{\GG}(U,U)$

If now $U$ and $\tilde{U}$ are arbitrary, one can consider $U\oplus\tilde{U}$. Observe that then $t\in\Hom_{\GG}(U,\tilde{U})$ if and only if 
\[T=\begin{pmatrix}
  0 & 0\\
  t & 0
  \end{pmatrix} \in \Hom_{\GG}(U\oplus\tilde{U},U\oplus\tilde{U})
\]
To conclude \eqref{eq:promintertwin}, we apply the first part of the proof to $U\oplus\tilde{U}$ and $T$. The only things that one needs to verify is that $(U\oplus\tilde{U})^{\pi^r_i}=U^{\pi^r_i}\oplus\tilde{U}^{\pi^r_i}$ (which is clear in view of \eqref{eq:corepsum}) and that the block form of $T$ remains after we restrict $U\oplus\tilde{U}$ to any $\HH_i$, which is again obvious.

To prove the other implication, let us pick $1\in I$ and consider $\HH_1\subset\GG$ and $\HH_2=\overline{\bigl\langle\bigcup_{i\in I\setminus\{1\}}\HH_j\bigr\rangle}$. It is now enough to take \eqref{eq:promintertwin} with $U=\tilde{U}=\ww$ to arrive at (the commutant of) \eqref{eq:promtheta3}. We are done thanks to \autoref{thm:promtheta}. 
\end{proof}

Let us note the following
\begin{corollary}\label{cor:prominv}
 Assume that $\GG$ is generated by $(\HH_i)_{i\in I}$ and let $U\in\B{H}{U}\vnt\Linf{G}$ be a representation of $\GG$. If $\mathcal{K}\subseteq\mathcal{H}$ is preserved by each $\HH_i$, then it is preserved by $\GG$.
\end{corollary}
\begin{proof}
$\mathcal{K}\subseteq\mathcal{H}$ is preserved by each $\HH_i$ if and only if $p_{\mathcal{K}}\in\Mor_{\HH_i}(U^{\pi^r_i},U^{\pi^r_i})$. We are done thanks to \autoref{thm:promintertwin}.
\end{proof}

Now let us remark the following:
\begin{theorem}
Let $\mathbb{G}$ be a compact quantum group.
\begin{enumerate}
 \item The quantum group $\GG$ is topologically generated by the quantum subgroups $\HH_1$ and $\HH_2$ (in the sense of \cite{BCV}) if and only if $\GG=\overline{\langle\HH_1,\HH_2\rangle}$ (in the sense of \autoref{def:twosubgroupsgenerateJKS}).
 \item The family of Hopf quotients $C^u(\GG)\to C^u(\HH_i)$ forms a jointly full family (in the sense of \cite{ChirvaRFD}) if and only if $\GG=\overline{\langle\bigcup_{i\in I}\HH_i\rangle}$ (in the sense of \autoref{def:twosubgroupsgenerateJKS}).
\end{enumerate}
\end{theorem}
\begin{proof}
 \begin{enumerate}
  \item The proof was outlined in \cite[Remark 3]{BCV} and relies on \autoref{thm:promintertwin} which extends \cite[Corollary 8.2]{BB10} outside the compact realm.
  \item The way \emph{joint fullness} is phrased in \cite[Definition 2.15]{ChirvaRFD} is precisely the conclusion of \autoref{thm:promintertwin}.
 \end{enumerate}
 \end{proof}
\section{Examples and applications}\label{sec:examples}
\subsection{Classical subsets in quantum groups}
Fix a morphism $\beta\in \Mor(C_0^u(\mathbb{G}), \mathnormal{B})$ and assume that the $C^{\ast}$-algebra $\mathnormal{B}$ is commutative. Then from general theory of $C^{\ast}$-algebras it follows that $\beta$ factors through the abelianization of $C_0^u(\mathbb{G})$. The latter was identified in \cite{daws,KN13} as a closed quantum subgroup of $\GG$: the group of characters of $\GG$, or the maximal classical subgroup of $\GG$, which we denote by $G$. Thus the Hopf image of the morphism $\beta$ is not bigger than $G$, and it is precisely the subgroup of $G$ generated by the image of the spectrum of $\mathnormal{B}$ under the Gelfand dual map of $\beta$, as one can easily check (\cite[Theorem 3.5]{PJphd}). 

\subsection{Quantum \texorpdfstring{``$az+b$''}{``az+b''} groups}

The quantum ``$az+b$'' groups were introduced in \cite{azb} and \cite{nazb} are quantum deformations of the group of affine transformations of $\mathbb{C}$. The construction of a quantum ``$az+b$'' group begins with a choice of a complex deformation parameter $q$ from a certain set (see \cite{nazb}). The parameter determines a multiplicative subgroup $\Gamma$ of $\mathbb{C}\setminus\{0\}$ and we let $\overline{\Gamma}$ be the closure of $\Gamma$ in $\mathbb{C}$ which is $\Gamma\cup\{0\}$. The $C^{\ast}$-algebra $\Cnot{G}$ is then isomorphic to the crossed product $C_0(\overline{\Gamma})\rtimes\Gamma$ with the action of $\Gamma$ on $\overline{\Gamma}$ given by multiplication of complex numbers. Let $b$ be the image under the natural morphism $C_0(\overline{\Gamma})\to{C_0(\overline{\Gamma})\rtimes\Gamma}=\Cnot{G}$ of the element $z$ affiliated with $C_0(\overline{\Gamma})$ given by $z(\gamma)=\gamma$ for all $\gamma\in\overline{\Gamma}$. Then $b$ is normal and its spectrum is equal to $\overline{\Gamma}$. Furthermore let $\{U_\gamma\}_{\gamma\in\Gamma}$ be the family of unitaries in $\mult{\Cnot{G}}$ implementing the action of $\Gamma$. Then for any $\gamma\in\Gamma$ we have $U_\gamma{b}U_{\gamma}^*=\gamma{b}$. 

The quantum group $\GG$ is coamenable (\cite[Section 6.2]{nazb}), so $\Cnot{G}=\Cunot{G}$ and so any morphism $\beta\in\Mor(\Cunot{G},\mathnormal{B})$ is determined by a covariant representation of the dynamical system $(\overline{\Gamma},\Gamma)$. Using this one can show that $\beta$ is either injective, or it has commutative image. In the latter case the Hopf image is a classical group, so there is no generating morphism $\beta$ which is not an isomorphism. In the dual language we can formulate this by saying that the quantum ``$az+b$'' group does not contain a proper subset which generates it.

\subsection{Double groups}

The procedure of constructing so called ``double groups'' in the $C^{\ast}$-algebraic framework goes back to the paper \cite{qlorentz} where a quantum deformation of the group $\mathrm{SL}(2,\mathbb{C})$ was constructed via the double group construction applied to the quantum $\mathrm{SU}(2)$ group. The construction which we will very briefly recall below yields always a non-compact locally compact quantum group and various examples have been presented in literature (see e.g.~\cite{double}). Let us note that in \cite{MNW03} the authors use an alternative name for the double group construction, namely ``quantum codouble''. 

The very rough view of the construction of the double group built over a locally compact quantum group $\KK$ is as follows: the $C^{\ast}$-algebra $\Cnot{G}$ is defined as $\Cnot{K}\otimes\Cnothat{K}$ with comultiplication $\Delta_\GG$ defined on simple tensors by
\[
\Delta_\GG(a\otimes{b})=\ww^{\KK}_{23}\Delta_{\KK}(a)_{13}\Delta_{\widehat{\KK}}(b)_{24}{\ww^{\KK}_{23}}\!^*.
\]
It turns out that $\GG$ defined by $(\Cnot{G},\Delta_\GG)$ is a locally compact quantum group. Moreover $\KK$ and $\widehat{\KK}$ are both closed quantum subgroups of $\GG$. If we let $\gamma_\KK$ and $\gamma_{\widehat{\KK}}$ be the corresponding embeddings $\Linfhat{K}\to\Linfhat{G}$ and $\Linf{K}\to\Linfhat{G}$ then by \cite[Section 4]{yamanouchi} (see in particular \cite[Corollary 4.12 and Formula (QD7)]{yamanouchi}) $\Linfhat{G}$ is generated by the images of $\gamma_\KK$ and $\gamma_{\widehat{\KK}}$. It follows from \autoref{prop:twosubgroups} that $\GG$ is generated by $\KK$ and $\widehat{\KK}$.

\subsection{Rieffel deformations}
Rieffel deformation is a procedure which allows one to deform a locally compact group into a quantum group. In order to perform it one needs to fix an abelian closed subgroup $\Gamma$ of a locally compact group $G$ and a 2-cocycle $\Psi$ on $\widehat{\Gamma}$. The resulting quantum group is denoted by $\GG^{\Psi}$, see \cite{PKRieffel,Rieffel93,Rieffel95,EnockVainerman} for further informations on Rieffel deformations. As noted in \cite[Section 3]{PKRieffel}, if $\Gamma\subset H\subset G$ is a chain of closed subgroups, then $\HH^{\Psi}\subset\GG^{\Psi}$. 

Now suppose that $\overline{\langle H_1,H_2\rangle}=G$ and that $\Gamma\subset H_1\cap H_2$. It follows from \cite[Theorem 4.12]{PKRieffelCP} and \autoref{prop:twosubgroups} that $\overline{\langle\HH_1^{\Psi},\HH_2^{\Psi}\rangle}=\GG^{\Psi}$. A particular example of this construction is obtained by taking $G=\mathrm{SL}(2,\mathbb{C})$, $H_1$ -- the group of upper-triangular matrices and $H_2$ are the lower-triangular matrices; $\Gamma=H_1\cap H_2\cong\mathbb{C}^{\times}$ is the group of diagonal matrices. For the 2-cocycle $\Psi$ one can take a cocycle described in \cite[Section 5]{PKRieffelCP}.

\section*{Acknowledgement}
 PJ, PK and PMS were partially supported by the NCN (National Center of Science) grant no. 2015/17/B/ST1/00085. PJ would like to express his gratitude to Alex Chirvasitu for asking questions that led to the answers contained in \autoref{sec:intertwiners}, as well as polishing the presentation of these results compared to the one given in \cite[Section 3.2.3]{PJphd}.
\bibliographystyle{alpha}
\bibliography{JKS}

\end{document}